\documentclass[10pt,reqno]{amsart}
\usepackage{amsfonts}
\usepackage{url,hyperref,multirow}
\usepackage{color}
\usepackage[dvipsnames]{xcolor}
\usepackage{cases}
\usepackage{subfigure}
\usepackage[ruled,vlined]{algorithm2e}
\usepackage{algorithmic}
\usepackage{cancel}
\usepackage{enumitem}
\definecolor{RED}{rgb}{1,0,0}

\usepackage{epsfig,amsmath,bm,caption2,epstopdf}
\usepackage{amssymb,version,graphicx,fancybox,mathrsfs,pifont,booktabs,mathtools}
\usepackage[marginparwidth=2.5cm, marginparsep=3mm]{geometry}

\catcode`\@=11 \theoremstyle{plain}
\@addtoreset{equation}{section}   

\@addtoreset{figure}{section}
\renewcommand\thefigure{\thesection.\@arabic\c@figure}
\@addtoreset{table}{section}
\renewcommand\thetable{\thesection.\@arabic\c@table}

\newtheorem{thm}{\bf Theorem}[section]
\newtheorem{cor}{\bf Corollary}[section]

\newtheorem{lmm}{\bf Lemma}[section]

\newenvironment{lemma}{\begin{lmm}}{\end{lmm}}
\theoremstyle{remark}
\newtheorem{remark}{\bf Remark}[section]
\theoremstyle{definition}

\newtheorem{exm}{\bf Example} 

\newcommand{\bs}[1]{\boldsymbol{#1}}

\begin{document}
\baselineskip 14pt

\title[Second-Order \(H^1\)-Norm Error Analysis]{Second-order $H^1$-norm error analysis for time-fractional advection-dispersion equations based on the fast averaged L1 method}
\author[
    L. Huang,\;  L. Li,\; P. Xie, \; M. Zhou
	]{Liangcai Huang${}^{\dag}$,\; Lin Li${}^{\ddag,*}$,\; and PENGCHENG XIE${}^{\S,*}$
		}
	\thanks{${}^{\dag}$Division of Mathematical Sciences, School of Physical and Mathematical Sciences, Nanyang Technological University, 637371, Singapore. (LHUANG037@e.ntu.edu.sg).\\
		\indent ${}^{\ddag}$School of Mathematics and Physics, University of South China, Hengyang 421001, China (lilinmath@usc.edu.cn).\\
        \indent ${}^{\S}$Applied Mathematics and Computational Research Division, Lawrence Berkeley National Laboratory, 1 Cyclotron Road, Berkeley, 94720, CA, USA (pxie@lbl.gov).\\
        \indent ${}^{*}$Corresponding author.
        }

\begin{abstract}
In this paper, based on the fast averaged L1 method, we present an error analysis for time-fractional advection-dispersion equations with a weak singularity at the initial time. An integrating-factor transformation is introduced to convert the tempered fractional derivative into the standard Caputo derivative, which is more suitable for discretization using the fast averaged L1 method. A sum-of-exponentials approximation is then incorporated into the averaged L1 method to reduce computational cost and storage while preserving the desired accuracy. By deriving error estimates for the discrete coefficients and the accumulated truncation errors, we establish the stability and $H^1$-norm convergence analysis, with a convergence order higher than those in the published literature. Numerical examples are tested to validate our theoretical results. The effects of the fractional parameters $\alpha$ and $\lambda$ on the solution are discussed. The memory effect and long-time tail phenomenon, which are known to exist in real  systems yet cannot be captured by classical integer-order equations,  are again found in the current fractional case.
\end{abstract}
 \maketitle

\vspace*{-10pt}
\section{Introduction}
As seen in \cite{cartea2007,guo2019,stynes2017,zhang2014,zhu2019}, fractional partial differential equations (PDEs) have provided a powerful tool for simulating challenging phenomena such as anomalous diffusion, transport and memory effects. When the fractional derivative is introduced into advection-dispersion models, the core significance lies in overcoming the limitations of traditional integer-order models based on Brownian motion, thereby accurately characterizing the ubiquitous phenomenon of anomalous diffusion in nature and engineering. For integer-order diffusion models, the mean squared displacement scales linearly with time, indicating that these models fail to describe behaviors such as the heavy tail in pollutant transport through porous media, sub-diffusion in drug release from biological tissues, or the unexpectedly rapid super-diffusion jumps in turbulence and plasmas. While for fractional-order models, by virtue of their nonlocality and memory effects (i.e., the time-fractional derivative captures the system's dependence on its historical states, and the space-fractional derivative accounts for long-range, Lévy-flight-type correlations), these models can provide a physically unified description spanning the entire spectrum from sub-diffusion to super-diffusion. Consequently, research on fractional-order models not only helps offer a more reliable theoretical foundation for applications such as groundwater contamination prediction, controlled drug delivery design, and reservoir engineering, but also drives the development of anomalous transport mechanisms in mathematical modeling, scientific computing, and inverse problems, establishing the fractional derivative as a key tool in advection-dispersion theory in complex media. So far, a substantial amount of work has been done on the qualitative analysis of fractional PDEs and the numerical methods for the computation of their solutions, e.g., \cite{cao2020,chen2018,quan20232,shen2023,zheng20252,zhu2019}.

In this paper, we investigate the stability and convergence of the fast averaged L1 method on nonuniform time meshes for a time-fractional advection-dispersion equation (TFADE):
\begin{equation}\label{Equation}
\begin{cases}
\displaystyle
\partial_t u + {}_0^C D_t^{\alpha,\lambda} u
= \Delta u - \sum\limits_{j=1}^d \partial_{x_j}u + f(\boldsymbol{x},t),
\quad &(\boldsymbol{x},t)\in \Omega\times(0,T], \\
u(\boldsymbol{x},t) = 0,
\quad &(\boldsymbol{x},t)\in \partial\Omega\times(0,T], \\[5pt]
u(\boldsymbol{x},0) = \phi(\boldsymbol{x}),
\quad &\boldsymbol{x}\in \Omega,
\end{cases}
\end{equation}
where $\lambda>0$, $0<\alpha<1$, and $\Omega\subset\mathbb{R}^d$ $(d=1,2)$. As seen in \cite{SABZIKAR201514}, the Caputo tempered fractional derivative is defined as follows:
\begin{equation}\label{tempered}
{}_0^C D_t^{\alpha,\lambda} u
=
\frac{e^{-\lambda t}}{\Gamma(1-\alpha)}
\int_0^t \frac{(e^{\lambda s}u(s))'}{(t-s)^\alpha}\,ds,
\end{equation}
where $\Gamma(\cdot)$ denotes the gamma function. When $\lambda=0$, \eqref{tempered} becomes the standard Caputo fractional derivative, i.e., 
\begin{equation*}
{}_0^C D_t^{\alpha} u
=
\frac{1}{\Gamma(1-\alpha)}
\int_0^t \frac{u'(s)}{(t-s)^\alpha}\,ds.
\end{equation*}
Compared with the model in \cite{zheng2024}, the proposed time-fractional advection-dispersion equation (TFADE) presented in \eqref{Equation} exhibits several advantages. First, it incorporates the Caputo tempered fractional derivative, which  is capable of describing the transition from power-law tailing to exponential tailing. Second, the inclusion of a first-order derivative advection term allows the model to better simulate transport phenomena in heterogeneous media. However, the first-order derivative term is spatially non-self-adjoint, which introduces additional nonsymmetric terms into the error equation. This property greatly complicates the stability and convergence analysis for \eqref{Equation}. To address these difficulties, we propose an integrating-factor transformation that converts the tempered fractional derivative into the standard Caputo derivative and eliminates the first-order derivative term. 

As mentioned in \cite{stynes2017}, solutions of time-fractional diffusion equations often exhibit weak singularity at the initial time. This has motivated the development of numerical methods on nonuniform temporal meshes, with the aim of recovering optimal convergence rates in the presence of such initial weak singularity. Moreover, these numerical methods may be broadly classified into three main categories: L1, L2-1\(_\sigma\), and L2 methods. For instance, in \cite{morgado2019}, Morgado et al. applied the L1 method to a time-fractional advection-diffusion equation, and obtained a temporal convergence rate of order \(\min\{2-\alpha,r\alpha\}\) on graded meshes in the \(L^2\)-norm. Subsequently, Kopteva~\cite{Kopteva2019} analyzed the L1 method on graded meshes for two- and three-dimensional time-fractional diffusion equations, and established a temporal accuracy of order \(\min\{2-\alpha,r\alpha\}\) in the \(L^\infty\)-norm. For the L2-1\(_\sigma\) method,
Chen et al.~\cite{chen2019} 
derived a temporal convergence rate of order \(\min\{2,r\alpha\}\) in the \(L^2\)-norm based on a class of nonuniform temporal meshes. In \cite{zhu2019}, Zhu et al. developed and analyzed a L2 method for time-fractional diffusion equations, achieving order $3-\alpha$ in the $H^1$-norm. Notably, the L1 method is computationally simpler and requires weaker regularity than the L2-1$_\sigma$ and L2 methods. To further extend the application of the L1 method, Ji et al.~\cite{ji2020} proposed an averaged L1 method to improve its convergence order, achieving a second-order temporal convergence rate in the \(L^2\)-norm. In \cite{shen2023}, Shen et al. applied the averaged L1 method to time-fractional diffusion equations and obtained a second-order temporal convergence rate in the \(L^\infty\)-norm for \(\alpha\ge (5-\sqrt{17})/2\). Later, Zheng and Wang~\cite{zheng2024} developed a fast averaged L1 method for a time-fractional mobile--immobile diffusion problem and proved a temporal convergence rate of order \(\min\{2,3-2\alpha\}\) in the \(L^2\)-norm, where the fast averaged L1 method shows significant advantages in improving computational efficiency. Moreover, more interesting developments can be seen in \cite{huang2016,huang2018,liao2018,lin2007,wang2018,wang2023,ye2024,Zayernouri2024}. However, higher-order convergence analysis for the fast averaged L1 method has not been reported in the published literature. This motivates the current work, and the main contributions and distinctive features in this paper are summarized as follows:
\begin{itemize}[leftmargin=1.5em,itemsep=0.3em,topsep=0pt,partopsep=0pt,parsep=0pt]
\item An efficient integrating-factor transformation is introduced for \eqref{Equation}, which greatly facilitates the construction of the discretization scheme and the subsequent theoretical analysis;  
\item We derive estimates for the discrete coefficients and the truncation error of the fast averaged L1 method on a graded temporal mesh. These estimates are key to the rigorous theoretical analysis of a fast temporal semidiscrete scheme for \eqref{Equation}. 
\item We establish second-order $H^1$-norm convergence for a full discretization scheme of \eqref{Equation}. This theoretical result improves upon that in \cite{cao2020}.
\end{itemize}

The remainder of this paper is organized as follows. Section~\ref{sec:AL1} introduces an integrating-factor transformation to simplify \eqref{Equation} and presents the fast temporal discretization. Section~\ref{sec:analysis} provides the stability and $H^1$-norm convergence analysis for the fast temporal semidiscrete scheme. Section~\ref{sec:spectral_err} presents the full discretization scheme using the spectral collocation method in space and establishes its second-order convergence in the $H^1$-norm. Section~\ref{sec:numerical} validates the theoretical results with several numerical experiments. Section~\ref{sec:conclude} concludes the paper.

\section
{An integrating-factor transformation for \eqref{Equation} and the fast temporal discretization}
\label{sec:AL1}

In this section, we introduce an integrating-factor transformation for \eqref{Equation} to facilitate the error estimates in the following sections. We then present a fast temporal discretization of \eqref{Equation} using the fast averaged L1 method.

For \eqref{Equation}, the fast averaged L1 method cannot be applied directly to discrete ${}_0^C D_t^{\alpha,\lambda}u$ due to the presence of the parameter $\lambda$. We therefore introduce an integrating-factor transformation:
\begin{equation}\label{eq:transform}
v(\boldsymbol{x},t)=e^{\lambda t-\frac{1}{2}\sum\limits_{j=1}^{d}x_j}\,u(\boldsymbol{x},t).
\end{equation}
As a result, the Caputo tempered derivative becomes 
\vspace{-0.1cm}
\begin{equation}\label{tran_frac} 
{}_0^C D_t^{\alpha,\lambda}u
= e^{-\lambda t}\,{}_0^C D_t^\alpha(e^{\lambda t}u)
= e^{-\lambda t+\frac12\sum\limits_{j=1}^d x_j}\,{}_0^C D_t^\alpha v.
\end{equation}
The standard Caputo derivative acting on $v$ is then obtained. This allows us to construct the temporal discretization for ${}_0^C D_t^\alpha v$ rather than for ${}_0^C D_t^{\alpha,\lambda}u$. Based on \eqref{eq:transform}, we further derive
\begin{equation}\label{eq2026051204}
\partial_{x_j}u
=
e^{-\lambda t+\frac12\sum\limits_{j=1}^d x_j}
\bigl(\partial_{x_j}v+\tfrac{v}{2}\bigr),\qquad
\partial_{x_jx_j}u
=
e^{-\lambda t+\frac12\sum\limits_{j=1}^d x_j}
\bigl(\partial_{x_jx_j}v+\partial_{x_j}v+\tfrac{v}{4}\bigr). 
\end{equation}
Substituting \eqref{tran_frac}-\eqref{eq2026051204} into \eqref{Equation} yields
\vspace{0.1cm}
\begin{equation}\label{eq:v-problem}
\begin{cases}
\partial_t v+{}_0^C D_t^\alpha v
= \Delta v+\mu v+\widetilde f(\boldsymbol{x},t),
\quad &(\boldsymbol{x},t)\in \Omega\times(0,T],\\[3mm]
v(\boldsymbol{x},t)=0,
\quad &(\boldsymbol{x},t)\in \partial\Omega\times(0,T],\\
v(\boldsymbol{x},0)=\widetilde\phi(\boldsymbol{x})
:=e^{-\frac12\sum\limits_{j=1}^d x_j}\phi(\boldsymbol{x}),
\quad &\boldsymbol{x}\in\Omega,
\end{cases}
\end{equation}
where
\begin{equation*}
\mu:=\lambda-\frac{d}{4}, \qquad
\widetilde f(\boldsymbol{x},t)
:=
e^{\lambda t-\frac12\sum\limits_{j=1}^d x_j}
f(\boldsymbol{x},t).   \vspace{-0.05cm}  
\end{equation*}
\begin{remark}
{\em Based on the integrating-factor transformation \eqref{eq:transform},
\eqref{eq:v-problem} has two significant advantages in comparison with \eqref{Equation}: 1) the tempered fractional derivative ${}_0^C D_t^{\alpha,\lambda}u$ is converted to the standard Caputo derivative ${}_0^C D_t^\alpha v$, and it is more suitable for discretization using the fast averaged L1 method; 2) Compared with \eqref{Equation}, the advection term in \eqref{eq:v-problem} is eliminated. This greatly facilitates the subsequent stability and convergence analysis.}
\end{remark}
Next we focus on the fast temporal discretization for \eqref{eq:v-problem}. As noted in \cite[Chapter~6]{Brunner2004}, the graded temporal mesh is known to be more effective than the uniform mesh for computing the time-fractional advection-dispersion equation. Motivated by this, we employ the following graded temporal mesh:
\begin{equation}\label{eq:graded_mesh_fast}
t_n = T (\frac{n}{M})^{r}, \qquad n=0,1,\dots,M,\vspace{0.15cm}
\end{equation}
where the grading parameter $r\ge1$. To facilitate the subsequent derivation, we define
\vspace{0.1cm}
\begin{equation*}
\tau_n := t_n-t_{n-1},\qquad
\nabla_\tau v^n := \tfrac{1}{\tau_n}{(v^n-v^{n-1})},\qquad
v^{n-\frac12}:= \tfrac{1}{2}(v^n+v^{n-1}).\vspace{0.1cm}
\end{equation*}
Integrating both sides of \eqref{eq:v-problem} over $[t_{n-1}, t_n]$ and taking the average, we obtain
\vspace{0.1cm}
\begin{equation} \label{average} 
\nabla_\tau v^n
    +
    \frac{1}{\tau_n}\int_{t_{n-1}}^{t_n} {}_0^C D_t^{\alpha} v\, dt
	=
    \frac{1}{\tau_n}\int_{t_{n-1}}^{t_n} (\Delta v + \mu v)\, dt
    +
    \frac{1}{\tau_n}
\int_{t_{n-1}}^{t_n}
\widetilde f(\boldsymbol{x},t)\,dt,\quad 1\le n\le M. \vspace{0.1cm}
\end{equation}
 To discretize $\frac{1}{\tau_n}\int_{t_{n-1}}^{t_n} {}_0^C D_t^{\alpha} v\, dt$ in \eqref{average}, we use the fast averaged L1 approximation presented in \cite{zheng2024} as follows:
 \vspace{-0.05cm}
\begin{equation}\label{eq:fast_avgL1}
\frac{1}{\tau_n}\int_{t_{n-1}}^{t_n}{}_0^C D_t^\alpha v \,dt  \approx \sum\limits_{k=1}^n \widetilde a_{n,k}\nabla_\tau v^k,
\end{equation}
\begin{equation}\label{eq:fastavgL1_coeff}   
\begin{cases}
		\widetilde a_{n,k}
		=
		\frac{1}{\tau_n\Gamma(1-\alpha)}
		\int_{t_{n-1}}^{t_n}
		\int_{t_{k-1}}^{t_k}
		\sum\limits_{\ell=1}^{N_{\exp}}\varpi_\ell e^{-s_\ell(t-s)}
		\,ds\,dt,
		& 1\le k\le n-2,\\[14pt]
		\widetilde a_{n,k}=\frac{1}{\tau_n\Gamma(1-\alpha)}
	\int_{t_{n-1}}^{t_n}
	\int_{t_{k-1}}^{\min\{t,t_k\}}
	(t-s)^{-\alpha}\,ds\,dt,
		& k=n-1,n.
	\end{cases}\vspace{0.15cm}
\end{equation}
where $\Gamma(\cdot)$ denotes the gamma function. Moreover, as shown in \cite{jiang2017,zhu2019},
$N_{\exp}$ represents the number of exponential terms, $\varpi_\ell>0$ and $s_\ell>0$ ($\ell =1, 2, \ldots$) are the SOE weights and exponential parameters, respectively. For \eqref{eq:fast_avgL1}, we denote $\bar{\partial}_{\tau,\varepsilon}^\alpha v^n := \sum\limits_{k=1}^n \widetilde a_{n,k}\nabla_\tau v^k$, and rewrite it as 
\vspace{-0.1cm}
\begin{equation}\label{eq2026051203}
\bar{\partial}_{\tau,\varepsilon}^\alpha v^n = \frac{1}{\tau_n\Gamma(1-\alpha)}\sum\limits_{\ell=1}^{N_{\exp}}\varpi_\ell\mathcal{F}_\ell^n(v^n) + \sum\limits_{k=n-1}^n \widetilde a_{n,k}\nabla_\tau v^k,
\end{equation}
where
\begin{equation*}
\mathcal{F}_\ell^n(v^n)=\sum\limits_{k=1}^{n-2} \nabla_{\tau} v^k \int_{t_{n-1}}^{t_n} \int_{t_{k-1}}^{t_k} e^{-s_\ell(t-s)} \,ds\,dt. \vspace{0.15cm}
\end{equation*}
It is worth noting that $\mathcal{F}_\ell^n(v^n)$ can be evaluated recursively via
\begin{equation*}
\mathcal{F}_\ell^n(v^n) = \frac{\tau_{n-1}}{\tau_n}e^{-s_\ell\tau_{n}}\mathcal{F}_\ell^{n-1}(v^{n-1})+\nabla_{\tau} v^{n-2} \int_{t_{n-1}}^{t_n} \int_{t_{n-3}}^{t_{n-2}} e^{-s_\ell(t-s)} \,ds\,dt,\quad 4\leq n\leq M.  
\end{equation*}

\begin{remark}\label{remark1}
{\em Compared with \eqref{eq:fast_avgL1} or \eqref{eq2026051203}, the classical averaged L1 approximation, as shown in \cite{ji2020}, is 
\begin{equation}\label{eq2026061201}
\frac{1}{\tau_n}\int_{t_{n-1}}^{t_n} {}_0^C D_t^{\alpha} v\, dt \approx \bar{\partial}_\tau^\alpha v^n := \sum\limits_{k=1}^{n} a_{n,k}\nabla_\tau v^k, 
\end{equation}
where
\begin{equation*}
a_{n,k} = \frac{1}{\tau_n\Gamma(1-\alpha)}
\int_{t_{n-1}}^{t_n}
\int_{t_{k-1}}^{\min\{t,t_k\}}(t-s)^{-\alpha}\,ds\,dt,
\quad 1\le k\le n. \vspace{0.15cm}
\end{equation*}
$\bar{\partial}_{\tau,\varepsilon}^\alpha v^n$ in \eqref{eq2026051203} offers a fast advantage over $\bar{\partial}_\tau^\alpha v^n$ in \eqref{eq2026061201}, i.e., only $\mathcal{O}(1)$ operations are required to update $\mathcal{F}_\ell^{\,n}(v^{n})$ given $\mathcal{F}_\ell^{\,n-1}(v^{n-1})$. Furthermore, as noted in \cite{zheng2024}, for sufficiently large $M$, the storage requirement and computational cost for evaluating $\bar{\partial}_{\tau,\varepsilon}^\alpha v^n$ are significantly lower than those for evaluating $\bar{\partial}_\tau^\alpha v^n$. Specifically, evaluating $\bar{\partial}_{\tau,\varepsilon}^\alpha v^n$ requires $\mathcal{O}(\log M)$ storage and $\mathcal{O}(M\log M)$ operations, whereas evaluating $\bar{\partial}_\tau^\alpha v^n$ requires $\mathcal{O}(M)$ storage and $\mathcal{O}(M^2)$ operations. }
\end{remark}

Now we focus on local truncation errors (i.e., $R_1^n$ and $R_2^n$) from the fast averaged L1 approximation and the trapezoidal rule as follows:
\begin{equation}\label{eq:error1}
R_1^n := \bar \partial_{\tau,\varepsilon}^\alpha v^n  - \frac{1}{\tau_n}\int_{t_{n-1}}^{t_n} {}_0^C D_t^{\alpha} v\, dt
, 
\end{equation}
\vspace{-0.15cm}
\begin{equation}\label{eq:error2}
R_2^n := \frac{1}{\tau_n}\int_{t_{n-1}}^{t_n} (\Delta v + \mu v)\, dt- (
\Delta v^{n-\frac{1}{2}}
+ \mu v^{n-\frac{1}{2}})
.  \vspace{0.1cm}
\end{equation}
To facilitate the presentation of an estimate for $R_1^n$ and $R_2^n$, as well as the analysis in the subsequent sections, the following preliminaries are provided.
\begin{itemize}[leftmargin=2em]
    \item The space $L^2(\Omega)$ is defined by
    \begin{equation*}
    L^2(\Omega):=\left\{u: u \text { is measurable on } \Omega \text { and }\|u\|<\infty\right\},
    \end{equation*}
    equipped with the norm and inner product
    \begin{equation}\label{l2_norm}
        \| u\|=(u,u)^{\frac{1}{2}},\qquad (u,v)=\int_{\Omega}u({\bs x})v({\bs x})\,d{\bs x},\qquad \forall u,v\in L^2(\Omega).
    \end{equation}
    \item  The space $H^m(\Omega)$ with $m \in \mathbb{N}$ is the space of functions $u \in L^2(\Omega)$ such that all the distributional derivatives of order up to $m$ can be represented by functions in $L^2(\Omega)$. That is,
    \begin{equation*}
        H^m(\Omega)=\left\{u \in L^2(\Omega): D^{\boldsymbol{\alpha}} u \in L^2(\Omega) \;\;\mathrm { for }\; \;0 \leq|\boldsymbol{\alpha}| \leq m\right\},
    \end{equation*}
    equipped with the norm 
\begin{equation}\label{hm_norm}
    \|u\|_{m}=\left(\sum_{|\boldsymbol{\alpha}|=0}^m\left\|D^{\boldsymbol{\alpha}} u\right\|^2\right)^{1 / 2},\qquad D^{\boldsymbol{\alpha}}=\frac{\partial^{|\boldsymbol{\alpha}|}}{\partial x_1^{\alpha_1} \ldots \partial x_d^{\alpha_d}},
    \end{equation}
    where  \(\boldsymbol{\alpha}=(\alpha_1,\ldots,\alpha_d)\) is a multiindex of order  $|\boldsymbol{\alpha}| =\sum
    _{i=1}^d \alpha_i$. Moreover, 
    \begin{equation*}
        H_0^1(\Omega)= \left\{u \in H^1(\Omega): u|_{\partial \Omega}=0\right\}.
    \end{equation*} 
\end{itemize}

An estimate for $R_1^n$ and $R_2^n$, given in \eqref{eq:error1} and \eqref{eq:error2}, is presented in the following lemma.
\begin{lemma}[\cite{shen2023,zheng2024}]\label{average_lem}
Suppose that $C>0$ is a constant, $\delta\ge\alpha$ and that for $0<t\le T$,
\begin{equation*}
\|\partial_t^\ell v(t)\|_{m} \le C\bigl(1+t^{\delta-\ell}\bigr),\qquad \ell=0,1,2,\quad m\ge 1.
\end{equation*}
Then, for $1\le n\le M$, there exists a positive constant $C_1$, independent of $M$, such that
\begin{equation}\label{Psi} 
\|R_1^n\|_{m} \le C_1 
\begin{cases}
\varepsilon + M^{-2} t_n^{\delta-\alpha-\frac{2}{r}},
& r(1+\delta)>2, \\[2mm]
\varepsilon + M^{-r(1+\delta)} t_n^{-1-\alpha} \ln(n+1),
& r(1+\delta)=2, \\[2mm]
\varepsilon + M^{-r(1+\delta)} t_n^{-1-\alpha},
& r(1+\delta)<2,
\end{cases} \vspace{0.15cm} 
\end{equation}
\begin{equation}\label{average_e2}
\|R_2^n\|_{m} \le C_1
\begin{cases}
M^{-r(\delta+1)}, & n=1,\\[2mm]
M^{-2} t_n^{\delta-\frac{2}{r}}, & n\ge 2,
\end{cases} \vspace{0.15cm}
\end{equation}
where $\varepsilon>0$ is a prescribed tolerance and $\|\cdot\|_m$ is defined in \eqref{hm_norm}. 
\end{lemma}

\section{Stability and $H^1$-norm convergence analysis for a semidiscrete scheme}\label{sec:analysis}

In this section, we focus on the stability and $H^1$-norm convergence analysis for a semidiscrete scheme. Based on Lemma \ref{average_lem}, $\frac{1}{\tau_n}\int_{t_{n-1}}^{t_n} {}_0^C D_t^{\alpha} v\, dt$ and $\frac{1}{\tau_n}\int_{t_{n-1}}^{t_n} (\Delta v + \mu v)\, dt$ given in \eqref{eq:error1}-\eqref{eq:error2} are well approximated by $\bar \partial_{\tau,\varepsilon}^\alpha v^n$ and $\Delta v^{n-\frac{1}{2}}
+ \mu v^{n-\frac{1}{2}}$. Therefore, \eqref{average} becomes 
\begin{equation}\label{semi} 
		\nabla_\tau v^n+\bar{\partial}_{\tau,\varepsilon}^\alpha v^n
		= \Delta v^{n-\frac{1}{2}}+\mu v^{n-\frac{1}{2}}+\bar{f}^{n-\frac{1}{2}},
		\quad 1\le n\le M,\vspace{0.1cm}
\end{equation}
where $
\bar{f}^{\,n-\frac{1}{2}}
:=
\frac{1}{\tau_n}
\int_{t_{n-1}}^{t_n}
\widetilde f(\boldsymbol{x},t)\,dt$.
Here \eqref{semi} is called the semidiscrete scheme. We start with the following two lemmas, which will be used in the subsequent stability and $H^1$-norm convergence analysis for \eqref{semi}.

\begin{lemma}\label{coe2} Suppose that $\varepsilon \le c_{\alpha,r} T^{-\alpha}M^{-\alpha}$ and 
\vspace{0.1cm}
\begin{equation}\label{set1}
    \frac{(2-\alpha) 2^{1-\alpha}-1}{\Gamma(3-\alpha)} \left(\frac{rT}{M}\right)^{1-\alpha}+C^*\frac{rT}{M} \le 1,\vspace{0.1cm}
\end{equation}
then the following inequality holds:
\begin{equation}\label{eq:coe2}
\frac{\widetilde a_{n,k}}{\tau_{k}}\le
\begin{cases}
    \frac{\widetilde a_{n,k+1}}{\tau_{k+1}},\quad &1\le k\le n-2, \\[10pt]
    \frac{1}{\tau_n} - C^* + \frac{\widetilde a_{n,k+1}}{\tau_{k+1}} ,\quad&k = n-1,
\end{cases}
\end{equation}
where \(\varepsilon>0\) denotes a prescribed tolerance, the same as that in \eqref{Psi}, \(C^*\) is an arbitrarily prescribed real number, and \(c_{\alpha,r}\) is a positive constant depending only on \(\alpha\) and \(r\).
\end{lemma}

\begin{proof}
To present our proof more clearly, the cases $1\leq k \leq n-3$, $k=n-2$ and $k = n-1$ are proved separately as follows:

\noindent\textbf{Case 1 ($1\le k\le n-3$).} We denote $h(x)=\sum\limits_{\ell=1}^{N_{\exp}}\varpi_\ell e^{-s_\ell x}/\Gamma(1-\alpha)$. Using \eqref{eq:fastavgL1_coeff} and the mean value theorem, we derive 
\vspace{0.1cm}
\begin{equation}\label{LHS}
\begin{split}
    \frac{\widetilde{a}_{n,k}}{\tau_{k}}&=\frac{1}{\tau_n\tau_{k}}\int_{t_{n-1}}^{t_n}\int_{t_{k-1}}^{t_{k}} h(t-s)\,ds\,dt=\frac{1}{\tau_n}\int_{t_{n-1}}^{t_n} h(t-\xi_k)\,dt,\vspace{0.15cm}
\end{split}
\end{equation}
where $\xi_k\in [t_{k-1},t_{k}]$. Since $h(x)$ is monotonically nonincreasing, we obtain \eqref{eq:coe2} from \eqref{LHS} directly.

\noindent\textbf{Case 2 ($k = n-2$).} 
Based on \eqref{eq:fastavgL1_coeff} , we derive
\begin{equation}\label{4.41}
\begin{split}
    \tfrac{\widetilde{a}_{n,n-2}}{\tau_{n-2}}=\;&\frac{1}{\tau_n\tau_{n-2}\Gamma(1-\alpha)}\int_{t_{n-1}}^{t_n}\int_{t_{n-3}}^{t_{n-2}} \sum\limits_{\ell=1}^{N_{\exp}}\varpi_\ell e^{-s_\ell (t-s)}\,ds\,dt\\[5pt]
    =\;&\tfrac{1}{\tau_n\tau_{n-2}\Gamma(1-\alpha)}\int_{t_{n-1}}^{t_n}\int_{t_{n-3}}^{t_{n-2}}(\sum\limits_{\ell=1}^{N_{\exp}}\varpi_\ell e^{-s_\ell (t-s)}-(t-s)^{-\alpha})\,ds\,dt\\[5pt]
    \;&+\frac{1}{\tau_n\tau_{n-2}\Gamma(1-\alpha)}\int_{t_{n-1}}^{t_n}\int_{t_{n-3}}^{t_{n-2}} (t-s)^{-\alpha}\,ds\,dt.\\[3pt]
\end{split}
\end{equation}
As shown in \cite{jiang2017,zhu2019}, for a prescribed tolerance $\varepsilon>0$, $
|\sum\limits_{\ell=1}^{N_{\exp}} \varpi_\ell e^{-s_\ell (t-s)} -(t-s)^{-\alpha}|
\le \varepsilon$ holds for $t-s\in[\tau_1,T]$.  Applying this to \eqref{4.41} yields
\vspace{0.1cm}
\begin{equation}\label{549}
    \frac{\widetilde{a}_{n,n-2}}{\tau_{n-2}} \le \frac{\varepsilon}{\Gamma(1-\alpha)} +\frac{1}{\tau_n\tau_{n-2}\Gamma(1-\alpha)}\int_{t_{n-1}}^{t_n}\int_{t_{n-3}}^{t_{n-2}} (t-s)^{-\alpha}\,ds\,dt. \vspace{0.15cm}
\end{equation}
With \eqref{eq:fastavgL1_coeff} and \eqref{549}, it holds
\vspace{0.1cm}
\begin{equation}\label{528}
\begin{split}
    \frac{\widetilde{a}_{n,n-2}}{\tau_{n-2}}\le \;&\frac{\varepsilon}{\Gamma(1-\alpha)} +\frac{1}{\tau_n\tau_{n-2}\Gamma(1-\alpha)}\int_{t_{n-1}}^{t_n}\int_{t_{n-3}}^{t_{n-2}} (t-s)^{-\alpha}\,ds\,dt\\[5pt]
    &-\frac{1}{\tau_n\tau_{n-1}\Gamma(1-\alpha)}\int_{t_{n-1}}^{t_n}\int_{t_{n-2}}^{t_{n-1}} (t-s)^{-\alpha}\,ds\,dt+\frac{\widetilde{a}_{n,n-1}}{\tau_{n-1}}\\[5pt]
    \le\;&\frac{\varepsilon}{\Gamma(1-\alpha)} +\frac1{\Gamma(1-\alpha)}\int_0^{1}(\frac1{\tau_{n-2}}\int_0^{\tau_{n-2}}(t\tau_n+\tau_{n-1})^{-\alpha}\,ds.\\[5pt]
&-\int_0^{1}(t\tau_n+s\tau_{n-1})^{-\alpha}\,ds)\,dt+\frac{\widetilde{a}_{n,n-1}}{\tau_{n-1}}.
\end{split}
\end{equation}
Based on the mean value theorem, we can further derive from \eqref{528}
\begin{equation}\label{609}
\scalebox{1.2}{$
\tfrac{\widetilde{a}_{n,n-2}}{\tau_{n-2}}-\tfrac{\widetilde{a}_{n,n-1}}{\tau_{n-1}}\le \tfrac{\varepsilon}{\Gamma(1-\alpha)} +\tfrac{\tau_{n-1}^{-\alpha}}{\Gamma(1-\alpha)}( (\xi+1)^{-\alpha}-\int_0^1(\xi+s)^{-\alpha}\, ds), $}\vspace{0.1cm}
\end{equation}
where $\xi \in [0,\frac{\tau_n}{\tau_{n-1}}]$. Based on  \eqref{eq:graded_mesh_fast}, we have $\xi\le \frac{\tau_n}{\tau_{n-1}} \leq 3^{r-1}$ for $n\ge 3$. Consequently, for the last term on the right‑hand side of \eqref{609}, there exists a positive constant $c_{\alpha,r}$, depending only on $\alpha$ and $r$, such that
\begin{equation}\label{621}
    (\xi+1)^{-\alpha}-\int_0^1(\xi+s)^{-\alpha}\, ds\le -c_{\alpha,r}.
\end{equation}
With $\tau_{n-1}\le rTM^{-1}$ and $\varepsilon \le c_{\alpha,r} T^{-\alpha}M^{-\alpha}$, we derive \eqref{eq:coe2}  
 by substituting \eqref{621} into 
\eqref{609}.

\noindent\textbf{Case 3 ($k=n-1$).}
From \eqref{eq:fastavgL1_coeff}, we have
\begin{equation}\label{741}
    \frac{\widetilde{a}_{n, n}}{\tau_n}=\frac{\tau_n^{-\alpha}}{\Gamma(3-\alpha)} \qquad \mathrm{and}\qquad \frac{\widetilde{a}_{n, n-1}}{\tau_{n-1}}=\frac{\left(\tau_n+\tau_{n-1}\right)^{2-\alpha}-\tau_n^{2-\alpha}-\tau_{n-1}^{2-\alpha}}{\tau_n \tau_{n-1} \Gamma(3-\alpha)}.
\end{equation}
With the mean value theorem, we obtain from \eqref{741}
\vspace{0.1cm}
\begin{equation}\label{746}
    \frac{\widetilde{a}_{n, n-1}}{\tau_{n-1}} \leq \frac{(2-\alpha) 2^{1-\alpha}}{\Gamma(3-\alpha)} \tau_n^{-\alpha}.\vspace{0.1cm}
\end{equation}
Using \eqref{741}-\eqref{746} and $\tau_n\le rTM^{-1}$, we derive
\vspace{0.1cm}
\begin{equation}\label{754}
\begin{split}
    &\frac{\widetilde{a}_{n, n-1}}{\tau_{n-1}}-\frac{1}{\tau_n}+C^*-\frac{\widetilde{a}_{n, n}}{\tau_{n}}\\[5pt]
    \le\; &\frac{(2-\alpha) 2^{1-\alpha}-1}{\Gamma(3-\alpha)} \tau_n^{-\alpha}-\frac{1}{\tau_n}+C^*\\[6pt]
    \le\; &\frac{1}{\tau_n}\left(\tfrac{(2-\alpha) 2^{1-\alpha}-1}{\Gamma(3-\alpha)} \left(\tfrac{rT}{M}\right)^{1-\alpha}-1+rTC^*M^{-1}\right).
\end{split}
\end{equation}
Under assumption \eqref{set1}, the quantity in the last line of \eqref{754} is nonpositive. Thus, \eqref{eq:coe2} follows, completing the proof.
\end{proof}

An estimate for 
$R^n := R_1^n+R_2^n$ (see \eqref{eq:error1}-\eqref{eq:error2}) is then provided in the following lemma.
\begin{lemma}
\label{lem:accumulated_R}
Suppose that the assumptions of Lemma~\ref{average_lem} hold and $r\ge 4$. Then, we have
\begin{equation}\label{eq:R_accumulated_1}
\sum\limits_{n=1}^M \tau_n \|R^n\|_{m} \le C_2 (\varepsilon + M^{-2})   \quad  \mathrm{and} \quad   \sum\limits_{n=1}^M \tau_n \|R^n\|_{m}^2 \le C_2(\varepsilon+ M^{-2})^2,
\end{equation}
where $C_2$ is a constant independent of $M$ and $\|\cdot\|_m$ is defined in \eqref{hm_norm}.
\end{lemma}

\begin{proof}
Since $r\ge4$ and $\delta\ge\alpha > 0$ (see Lemma~\ref{average_lem}), we have
$r(1+\delta) > 4$. It follows from Lemma~\ref{average_lem} that
\begin{equation}
\label{eq:R_basic_bound}
\|R^n\|_{m} \le 
C_1
\begin{cases}
    \varepsilon + M^{-2} t_1^{\delta-\alpha-\frac{2}{r}} + M^{-r(\delta+1)}, & n = 1, \\[2mm]
    \varepsilon + M^{-2} t_n^{\delta-\alpha-\frac{2}{r}} + M^{-2} t_n^{\delta - \frac{2}{r}}, & n \ge 2.
\end{cases}
\end{equation}
When $n=1$, \eqref{eq:graded_mesh_fast} gives $t_1=TM^{-r}$, indicating that  
\vspace{0.05cm}
\begin{equation}\label{lem111}
M^{-r(\delta+1)}\le M^{-r(\delta+1)+r(1+\alpha)}= M^{-r(\delta-\alpha)}=T^{-\delta+\alpha+\frac{2}{r}}M^{-2} t_1^{\delta-\alpha-\frac{2}{r}} .
\end{equation}
Similarly, we further derive 
\begin{equation}\label{lem222}
\begin{split}
M^{-2} t_n^{\delta-\frac{2}{r}} \le T^{\alpha}M^{-2} t_n^{\delta-\alpha-\frac{2}{r}},     \quad\quad 
\forall\, n\ge2.
\end{split}
\end{equation}
Consequently, substituting \eqref{lem111}--\eqref{lem222} into \eqref{eq:R_basic_bound} yields
\begin{equation}\label{lem11}
\|R^n\|_{m} \le C_2(\varepsilon + M^{-2} t_n^{\delta-\alpha-\frac{2}{r}}),
\end{equation} 
where the constant $C_2$ is independent of $M$. Based on \eqref{eq:graded_mesh_fast}, we have 
\begin{equation}\label{eq:tau_bound}
\tau_n = t_n - t_{n-1} \le C M^{-r} n^{r-1}, \qquad 1 \le n \le M.
\end{equation}
where the constant $C$ is independent of $M$. With \eqref{eq:tau_bound} and $r(1 + \delta - \alpha) \ge r \ge 4$, we derive
\begin{equation}\label{lem1}
\sum\limits_{n=1}^M \tau_n M^{-2} t_n^{\delta-\alpha-\frac{2}{r}}
\le C_2 \sum\limits_{n=1}^M M^{-r} n^{r-1} M^{-2}(\tfrac{n}{M})^{r(\delta - \alpha) - 2} \le C_2 M^{-2}.
\end{equation}
With \eqref{lem11} and \eqref{lem1}, we obtain
\begin{equation}\label{lem3}
\sum\limits_{n=1}^M \tau_n \|R^n\|_{m} \le C_2( \varepsilon + M^{-2}). 
\end{equation}
Similarly, from \eqref{lem11}, we derive
\begin{equation}\label{lem22}
\begin{split}
\sum\limits_{n=1}^M \tau_n \|R^n\|_{m}^2
&\le C_2 \sum\limits_{n=1}^M \tau_n (\varepsilon^2 + 2\varepsilon M^{-2} t_n^{\delta-\alpha-\frac{2}{r}} + M^{-4} t_n^{2\delta-2\alpha-\frac{4}{r}}) \\
&\le C_2 ( T\varepsilon^2 + 2\varepsilon \sum\limits_{n=1}^M \tau_n M^{-2} t_n^{\delta-\alpha-\frac{2}{r}} + \sum\limits_{n=1}^M \tau_n M^{-4} t_n^{2\delta-2\alpha-\frac{4}{r}} ).
\end{split}
\end{equation}
With \eqref{eq:tau_bound} and $r + 2r(\delta-\alpha) \ge 4$, it holds
\begin{equation}\label{lem33}
\sum\limits_{n=1}^M \tau_n M^{-4} t_n^{2\delta-2\alpha-\frac{4}{r}}
\le C_2 \sum\limits_{n=1}^M M^{-r} n^{r-1} M^{-4} t_n^{2(\delta-\alpha)-\frac{4}{r}} \le C_2 M^{-4}.
\end{equation}
Substituting \eqref{lem1} and \eqref{lem33} into \eqref{lem22} yields
\begin{equation}\label{lem4}
\sum\limits_{n=1}^M \tau_n \|R^n\|_{m}^2 \le C_2(\varepsilon+ M^{-2})^2.\vspace{0.05cm}
\end{equation}
Finally, \eqref{eq:R_accumulated_1} can be proved by using \eqref{lem4} and \eqref{lem3}.
\end{proof}

Next, the stability and \(H^1\)-norm convergence analysis for \eqref{semi} are presented in the Theorems~\ref{thm:fast_stability} and~\ref{thm:semi}, respectively.

\begin{thm}\label{thm:fast_stability} Suppose that the assumptions of Lemma~\ref{coe2} hold and that \(0\le \rho<\rho_1\). Then, for $1\le n\le M$, the numerical solution of \eqref{semi} satisfies 
\begin{equation*}
\|v^n\|
\le
\begin{cases}
    e^ {\tfrac{\rho\rho_1t_n}{\rho_1-\rho}} 
( \|v^0\| + \tfrac{\rho_1}{\rho_1-\rho}  \sum\limits_{j=1}^n \tau_j \|\bar f^{\,j-\frac12}\|), & 0 < \rho < \rho_1,\\[15pt]
\|v^0\|
+
\sum\limits_{j=1}^n \tau_j\|\bar f^{\,j-\frac12}\|, &\rho=0,
\end{cases}
\end{equation*}
where $\|\cdot\|$ is defined in \eqref{l2_norm}, $C_P$ is the Poincaré constant,
\begin{equation*}
\rho_1 = \tfrac{2}{\max\limits_{1\le n\le M}\tau_n} \qquad \mathrm{and} \qquad \rho = \max\{\mu-C_P^{-2},0\}.  
\end{equation*}
\end{thm}

\begin{proof}
Let $a_\mu(u,w) := (\nabla u,\nabla w) - (\mu u,w)$. Multiplying \eqref{semi} by a test function $\phi$ yields
\vspace{0.1cm}
\begin{equation}\label{weak}
(\nabla_\tau v^n,\phi) + (\bar{\partial}_{\tau,\varepsilon}^\alpha v^n,\phi) + a_\mu(v^{n-\frac12},\phi) = (\bar f^{\,n-\frac12},\phi).\vspace{0.1cm}
\end{equation}
Based on the Poincaré inequality, we have 
\begin{equation*}
a_\mu(w,w) = \|\nabla w\|^2 - (\mu w,w) \ge -\rho\|w\|^2,  \qquad  \forall\, w\in H_0^1(\Omega).   
\end{equation*}
Let $\{\phi_m\}_{m = 1}^{\infty}$ be a $L^2$-orthonormal basis of eigenfunctions associated with $a_\mu$. Then we have 
\begin{equation*}
a_\mu(\phi_m,w)=\lambda_m^\mu(\phi_m,w),
\qquad \forall w\in H_0^1(\Omega),\vspace{0.1cm}
\end{equation*}
where $\lambda_m^\mu\ge -\rho$. Let
\begin{equation}\label{l2expan}
v^n = \sum\limits_{m\ge1} v_m^n \phi_m\qquad \mathrm{and}\qquad
\bar f^{\,n-\frac12} = \sum_{m\ge1} \bar f_m^{\,n-\frac12} \phi_m,
\end{equation}
indicating that \eqref{weak} becomes  
\begin{equation}\label{s1}
\frac{1}{{\tau_n}}{(v_m^n - v_m^{n-1})}
+ \sum\limits_{k=1}^n \frac{\widetilde a_{n,k}}{\tau_k}(v_m^k - v_m^{k-1})
+ \frac{1}{2}\lambda_m^\mu (v_m^n + v_m^{n-1})
= \bar f_m^{\,n-\frac12},
\end{equation}
where $\{\widetilde{a}_{n,k}\}_{k=1}^n$ are given by \eqref{eq:fastavgL1_coeff}. We rewrite \eqref{s1} as
\begin{equation}\label{s2}
\scalebox{1}{$
D_{m,n} v_m^n
= (D_{m,n} -\lambda_m^{\mu} - \frac{\widetilde a_{n,n-1}}{\tau_{n-1}})  v_m^{n-1}+ \sum\limits_{k=1}^{n-2}( \frac{\widetilde a_{n,k+1}}{\tau_{k+1}} - \frac{\widetilde a_{n,k}}{\tau_k}) v_m^k
+ \frac{\widetilde a_{n,1}}{\tau_1} v_m^0
+ \bar f_m^{\,n-\frac12},$}
\end{equation}
where
\begin{equation*}
D_{m,n} := \frac{1}{\tau_n} + \frac{\widetilde a_{n,n}}{\tau_n} + \frac{\lambda_m^\mu}{2}.\vspace{0.15cm}
\end{equation*}
Since $\lambda_m^\mu \ge -\rho>-\rho_1=\tfrac{2}{\max\limits_{1\le n\le M}\tau_n}$, we further derive
\vspace{0.1cm}
\begin{equation}\label{s4}
0 \leq \frac{1}{\tau_n} - \frac{1}{\max\limits_{1\le n\le M}\tau_n} < \frac{1}{\tau_n} + \frac{\widetilde a_{n,n}}{\tau_n} - \frac{\rho}{2}  \leq D_{m,n}.\vspace{0.1cm}
\end{equation}
By Lemma~\ref{coe2} with \(C^*=0\), we have
\vspace{0.1cm}
\begin{equation}\label{0.1}
  \frac{\widetilde a_{n,n-1}}{\tau_{n-1}} \leq \frac{1}{\tau_n}+\frac{\widetilde a_{n,n}}{\tau_{n}} \le D_{m,n} -\frac{\lambda_m^\mu}{2} \quad \mathrm{and} \quad
\frac{\widetilde a_{n,k}}{\tau_k} \le \frac{\widetilde a_{n,k+1}}{\tau_{k+1}}\, (1\le k\le n-2), \vspace{0.07cm}
\end{equation}
indicating that \(\frac{\widetilde a_{n,k+1}}{\tau_{k+1}} - \frac{\widetilde a_{n,k}}{\tau_k}\) given in \eqref{s2} are nonnegative. We define
\begin{equation}\label{0.2}
M_m^n := \max\limits_{0 \le j \le n} |v_m^j|.  
\end{equation}
Taking absolute values in \eqref{s2} and combining \(-\lambda_m^\mu \le \rho\) with \eqref{0.1} yields
\vspace{0.1cm}
\begin{equation}\label{s3}
\begin{split}
    D_{m,n} |v_m^n|
\le\; &\bigl( \bigl|D_{m,n} -\lambda_m^{\mu} - \frac{\widetilde a_{n,n-1}}{\tau_{n-1}} \bigl|+\frac{\widetilde a_{n,n-1}}{\tau_{n-1}} \bigl) M_m^{n-1}
+ |\bar f_m^{\,n-\frac12}|\\[3pt]
\le\; & \left(D_{m,n}+\rho\right)M_m^{n-1}
+ |\bar f_m^{\,n-\frac12}|.\\[3pt]
\end{split}
\end{equation}
It follows from \eqref{s4} that \(D_{m,n}\ge \frac{1}{\tau_n}-\frac{\rho}{2}\), which, together with \eqref{s3}, implies
\vspace{0.1cm}
\begin{equation}\label{821}
|v_m^n|
\le (1 + \frac{2\rho\tau_n}{2-\rho\tau_n}) M_m^{n-1}
+ \frac{2\tau_n}{2-\rho\tau_n}|\bar f_m^{\,n-\frac12}|.\vspace{0.15cm}
\end{equation}
Based on the notation $\rho_1=\frac{2}{\max\limits_{1\le n\le M}\tau_n}$, we derive from 
\eqref{0.2} and \eqref{821} 
\begin{equation*}
M_m^n = \max\{v_m^n,M_m^{n-1}\}
\le (1 + \frac{\rho\rho_1\tau_n}{\rho_1-\rho}) M_m^{n-1}
+ \frac{\rho_1\tau_n}{\rho_1-\rho} |\bar f_m^{\,n-\frac12}|. \vspace{0.1cm}
\end{equation*}
Using the discrete Gronwall inequality and $|M_m^0|=|v_m^0|$, we further derive 
\vspace{0.05cm}
\begin{equation} \label{m1}
v_m^n \le M_m^n
\le e^ {\tfrac{\rho\rho_1t_n}{\rho_1-\rho}} ( |v_m^0| + \frac{\rho_1}{\rho_1-\rho} \sum\limits_{j=1}^n \tau_j |\bar f_m^{\,j-\frac12}|).\vspace{0.1cm}
\end{equation}
We apply the Parseval's identity to \eqref{l2expan}. Together with \eqref{m1}, we have
\begin{equation}\label{0.3}
\|v^n\|
= ( \sum\limits_{m\ge1} |v_m^n|^2 )^{1/2} \le e^ {\tfrac{\rho\rho_1t_n}{\rho_1-\rho}} 
( \|v^0\| + \frac{\rho_1}{\rho_1-\rho}  \sum\limits_{j=1}^n \tau_j \|\bar f^{\,j-\frac12}\|).
\end{equation}
Setting $\rho = 0$ in \eqref{0.3} gives
\begin{equation*}
\|v^n\| \le \|v^0\| + \sum_{j=1}^n \tau_j \|\bar f^{\,j-\frac12}\|.   
\end{equation*}
This completes the proof.
\end{proof}

The $H^1$-norm convergence analysis for \eqref{semi} is now summarized in the following theorem.
\begin{thm}\label{thm:semi}
Suppose that the assumptions of Lemma~\ref{coe2} hold and $r\ge4$. Then, we have
\vspace{0.15cm}
\begin{equation*}\vspace{-0.15cm}
\scalebox{1.2}{$
\sum\limits_{k=1}^n\frac{\|e^k-e^{k-1}\|^2}{\tau_k} +\frac{\sigma(\alpha)}{\Gamma(3-\alpha)}\sum\limits_{k=1}^n\frac{\|e^k-e^{k-1}\|^2}{\tau_k^{\alpha}}+\|e^n\|_1^2
\leq C_3(\varepsilon+ M^{-2})^2,   $}\vspace{0.15cm}
\end{equation*}
where $v$ is the exact solution of \eqref{eq:v-problem}, $\{v^k\}_{k=1}^{M}$ is the numerical solution of \eqref{semi}, and the norms $\|\cdot\|$ and $\|\cdot\|_1$ are defined in \eqref{l2_norm} and \eqref{hm_norm}, respectively. The constant $C_2$ is the same as that in \eqref{eq:R_accumulated_1}.
Furthermore, we define
\begin{equation*}
e^k:=v(t_k)-v^k,\quad \sigma(\alpha)=2\alpha(1-2^{-\alpha}),\quad
 C_3=C_2|\mu|t_n e^ {\frac{\rho\rho_1t_n}{\rho_1-\rho}}(\tfrac{\rho_1}{\rho_1-\rho})^2   + 1 . 
\end{equation*}
\end{thm}
\begin{proof}
The error equation for \eqref{semi} is
\begin{equation}\label{weak1}
\begin{cases}
\nabla_\tau e^n + \bar{\partial}_{\tau,\varepsilon}^\alpha e^n
= \Delta e^{n-\frac12} + \mu e^{n-\frac12} + R^n,   \quad\, \text{in } \Omega,\\[4pt]
e^n|_{\partial\Omega} = 0,   \quad\quad
e^0|_{\overline{\Omega}} = 0.
\end{cases}
\end{equation}
Multiplying \eqref{weak1} by $\phi:= e^n - e^{n-1}$ yields
 \vspace{0.1cm}
\begin{equation}\label{c1}
\begin{split}
~   & \tfrac{1}{\tau_n}\|e^n - e^{n-1}\|^2
+ \bigl(\bar{\partial}_{\tau,\varepsilon}^\alpha e^n, e^n - e^{n-1}\bigr)
+ \tfrac{1}{2} \bigl(\|\nabla e^n\|^2 - \|\nabla e^{n-1}\|^2\bigl) \\[5pt]
\le\; &\tfrac{|\mu|}{2} \bigl(\|e^n\|^2 - \|e^{n-1}\|^2\bigr)
+ \tfrac12{\tau_n \|R^n\|^2}
+ \tfrac{1}{2\tau_n}\|e^n - e^{n-1}\|^2.
\end{split}
\end{equation}
As seen in \cite[Theorem 2.1]{zheng20252}, it holds
 \vspace{0.1cm}
\begin{equation}\label{eq:con1}
\frac{\sigma(\alpha)}{2\Gamma(3-\alpha)} \sum\limits_{k=1}^n \frac{1}{\tau_k^{\alpha}}{\|e^k - e^{k-1}\|^2} \leq \sum\limits_{k=1}^n \bigl(\bar{\partial}_{\tau,\varepsilon}^\alpha e^k, e^k - e^{k-1}\bigr).    \vspace{0.1cm}
\end{equation}
With \eqref{eq:con1}, summing up both sides of \eqref{c1} from $k = 1$ to $n$ gives
\vspace{0.15cm}
\begin{equation}\label{eq:sta1}
\begin{split}
~  &\sum\limits_{k=1}^n \tfrac{1}{\tau_k}{\|e^k - e^{k-1}\|^2}
+ \tfrac{\sigma(\alpha)}{2\Gamma(3-\alpha)} \sum\limits_{k=1}^n \tfrac{1}{\tau_k^{\alpha}}{\|e^k - e^{k-1}\|^2}
+ \tfrac12\big({\|e^n\|_1^2 - \|e^0\|_1^2}\big) \\[4pt]
\le\; &\tfrac{|\mu|}{2} \big({\|e^n\|^2 - \|e^0\|^2}\big)
+ \sum\limits_{k=1}^n \tfrac12{\tau_k \|R^k\|^2}
+ \sum\limits_{k=1}^n \tfrac{1}{2\tau_k}{\|e^k - e^{k-1}\|^2}.
\end{split}
\end{equation}
Since $e^0 = 0$, \eqref{eq:sta1} reduces to
\vspace{0.1cm}
\begin{equation}\label{eq:sta2}
\scalebox{1}{$
\sum\limits_{k=1}^n \frac{1}{2\tau_k}{\|e^k - e^{k-1}\|^2}
+ \frac{\sigma(\alpha)}{\Gamma(3-\alpha)} \sum\limits_{k=1}^n \frac{1}{\tau_k^{\alpha}}{\|e^k - e^{k-1}\|^2}
+ \|e^n\|_1^2
\le |\mu| \|e^n\|^2
+ \sum\limits_{k=1}^n \tau_k \|R^k\|^2.$}\vspace{0.1cm}
\end{equation}
Following the same argument as that of Theorem~\ref{thm:fast_stability}, we further derive from \eqref{weak1} that
\vspace{0.1cm}
\begin{equation}\label{eq:vn1}
\|e^n\|^2 \le  e^ {\frac{\rho\rho_1t_n}{\rho_1-\rho}} 
( \tfrac{\rho_1}{\rho_1-\rho} \sum\limits_{k=1}^n \tau_k \|R^k\| )^2
\le t_n e^ {\tfrac{\rho\rho_1t_n}{\rho_1-\rho}}(\frac{\rho_1}{\rho_1-\rho})^2 \sum\limits_{k=1}^n \tau_k \|R^k\|^2.\vspace{0.15cm}
\end{equation}
Substituting \eqref{eq:vn1} into \eqref{eq:sta2} yields
\vspace{0.15cm}
\begin{equation}\label{202606061}
\begin{split}
~ &\sum\limits_{k=1}^n \frac{\|e^k - e^{k-1}\|^2}{2\tau_k}
+ \frac{\sigma(\alpha)}{\Gamma(3-\alpha)} \sum\limits_{k=1}^n \frac{\|e^k - e^{k-1}\|^2}{\tau_k^{\alpha}}
+ \|e^n\|_1^2 \\[4pt]
\leq\; & ( |\mu|t_n e^ {\tfrac{\rho\rho_1t_n}{\rho_1-\rho}}(\tfrac{\rho_1}{\rho_1-\rho})^2   + 1 ) \sum\limits_{k=1}^n \tau_k \|R^k\|^2.
\end{split}
\end{equation}
Combining \eqref{202606061} with Lemma~\ref{lem:accumulated_R} completes the proof.
\end{proof}
\begin{remark}
From Theorem~\ref{thm:semi}, we obtain $\|e^n\|_1 \le \sqrt{C_3}(\varepsilon + M^{-2})$, which shows that the semidiscrete scheme \eqref{semi} achieves second-order temporal convergence in the $H^1$-norm. This convergence rate is higher than that of the method in \cite{cao2020}, as will be validated by numerical experiments in Section~\ref{sec:numerical}.
\end{remark}

\section{ $H^1$-norm convergence analysis to a full discretization scheme}
\label{sec:spectral_err}

In this section, based on \eqref{semi} and spectral collocation method in space,
a full discretization scheme for \eqref{eq:v-problem} is presented, and its convergence analysis is also provided.
Before presenting the full discretization scheme, the spectral collocation method given in \cite{Shen2011,Trefethen2000} is introduced, where we take the two-dimensional case as an illustrative example.
\begin{enumerate}[leftmargin=2em]
    \item ({\bf Legendre-Gauss-Lobatto nodes and weights}) Consider the rectangular domain
    \begin{equation*}
        \Omega :=(a,b)\times(c,d).
    \end{equation*}
    The collocation points along the $x-$axis are given by the affine transformation 
\begin{equation}\label{trans}
        x_i=\frac{b-a}{2}\xi_i+\frac{a+b}{2},
        \qquad 0\le i \le N,
    \end{equation}
    where $\{\xi_i\}^{N}_{i=0}$ represents the Legendre-Gauss-Lobatto (LGL) nodes along the $x-$axis. The weights are
\begin{equation}\label{eq2026060701}
     \omega_i
        =
        \frac{2}{N(N+1)}
        \frac{1}{[L_N(\xi_i)]^2},
        \qquad 0\le i\le N,   
    \end{equation}
    where \(L_N\) is the Legendre polynomial of degree \(N\). Similarly, the collocation points and weights along the $y-$axis can be obtained following \eqref{trans}-\eqref{eq2026060701}, i.e.,
\begin{equation*}
    y_j = \frac{d-c}{2}\xi_j + \frac{c+d}{2} \qquad \mathrm{and} \qquad  \omega_j
        =
        \frac{2}{N(N+1)}
        \frac{1}{[L_N(\xi_j)]^2},
\end{equation*}
where $\{\xi_j\}^{N}_{j=0}$ represents the Legendre-Gauss-Lobatto (LGL) nodes along the $y-$axis. 
    \vspace{0.1cm}
    \item ({\bf Differentiation matrix}) Let $D$ be the LGL differentiation matrix on \([-1,1]\), whose entries are given by
    \begin{equation*}
        (D)_{ij}
    =
    \begin{cases}
    \displaystyle
    \frac{L_N(\xi_i)}{L_N(\xi_j)}
    \frac{1}{\xi_i-\xi_j},
    & i\ne j,\\[5mm]
    \displaystyle
    -\frac{N(N+1)}{4},
    & i=j=0,\\[4mm]
    \displaystyle
    \frac{N(N+1)}{4},
    & i=j=N,\\[4mm]
    0,
    & i=j,\ 1\le i\le N-1.
    \end{cases}\vspace{0.05cm}
    \end{equation*}
     The corresponding second-order differentiation matrices on $\Omega$ along $x-$ and $y-$ axes are
    \begin{equation*}
        D_x^{(2)}
        =
        \left(\frac{2}{b-a}\right)^2 D^2
        \qquad \mathrm{and}  \qquad
        D_y^{(2)}
        =
        \left(\frac{2}{d-c}\right)^2 D^2.
    \end{equation*}
As a result, the Laplacian operator can be written as
\begin{equation*}
        \Delta_N
        :=
        I \otimes D_x^{(2)}
        +
        D_y^{(2)}\otimes I,
    \end{equation*}
where $\otimes$ represents the tensor-product and $I$ represents the identity matrix. 
    \item ({\bf Discrete 
 inner product}) Let $\mathbb Q_N(\Omega)$ be the tensor-product polynomial space of degree at
    most \(N\) in each variable. We denote
    $\mathbb Q_N^0(\Omega) := \mathbb Q_N(\Omega)\cap H_0^1(\Omega)$. For $\forall u,v\in \mathbb Q_N(\Omega)$, 
    the discrete inner product is 
\begin{equation}\label{norm_n}
 (u,v)_N :=
\frac{(b-a)(d-c)}{4}\sum_{i=0}^{N}\sum_{j=0}^{N}
u(x_i,y_j)v(x_i,y_j)\omega_{i}\omega_{j}.
\end{equation}
The reader is referred to \cite{Shen2011} for more details.

\end{enumerate}

Next, based on \eqref{semi} and the spectral collocation method described above, a full discretization scheme for \eqref{eq:v-problem} is obtained as follows:
\begin{equation}\label{eq:fast_scheme_weak}
(\nabla_\tau v_N^n,\phi_N)_N
    	+
(\bar{\partial}_{\tau,\varepsilon}^\alpha v_N^n,\phi_N)_N
    	+
    	(\nabla v_N^{n-\frac12},\nabla\phi_N)_N
    	=
    	(\mu v_N^{n-\frac12},\phi_N)_N+(\bar f^{\,n-\frac12},\phi_N)_N,\vspace{0.1cm}
    \end{equation}
    where $\forall \phi_N\in\mathbb Q_N^0(\Omega)$. Before presenting $H^1$-norm  convergence analysis for \eqref{eq:fast_scheme_weak}, the following preliminaries are provided.
\begin{itemize}[leftmargin=2em]
    \item ({\bf Orthogonal projection operator}) Let \(\pi_N^{1,0}:H_0^1(\Omega)\to \mathbb Q_N^0(\Omega)\) be the \(H_0^1\)-orthogonal projection operator, defined by
    \vspace{0.05cm}
    \begin{equation*}
    (\nabla(\varphi-\pi_N^{1,0}\varphi),\nabla \phi_N )=0,
    \quad
    \forall \phi_N\in \mathbb Q_N^0(\Omega).
    \end{equation*}
    \item ({\bf Approximation results}) As shown in \cite{bernardi11992,Shen2011},
     it holds
\begin{equation}\label{eq:spec1}
    \hspace{-5.95cm}(1)\;\|\varphi\| \le \|\varphi\|_N
    	\le C_d\|\varphi\|,
    	\quad
    	\forall \varphi\in \mathbb Q_N(\Omega),\vspace{0.15cm}
    \end{equation} 
    \begin{equation}\label{eq:spec3}
    \hspace{-3.52cm}(2)\;	\|\varphi-\pi_N^{1,0}\varphi\|_l
    	\le cN^{\,l-m}\|\varphi\|_{m},
    	\quad
    	\forall \varphi\in H^m(\Omega)\cap H_0^1(\Omega),\vspace{0.15cm}
    \end{equation}
    \begin{equation}\label{eq:spec2}
    (3)\;|(\varphi,v_N)-(\varphi,v_N)_N|
    	\le cN^{-m}\|\varphi\|_{m} \|v_N\|,
    	\quad
    	\forall \varphi\in H^m(\Omega),\
    	\forall v_N\in \mathbb Q_N(\Omega),\vspace{0.15cm}
    \end{equation}
    where $\|\cdot\|_N = (\cdot, \cdot)^{1/2}_{N}$,  $c$ is a constant independent of $N$, $m \ge 1$ , $l = 0, 1$ and \(C_d>0\) is a constant depending only on the spatial dimension.
\end{itemize}

The $H^1$-norm convergence analysis for \eqref{eq:fast_scheme_weak} is now summarized in the following theorem.

\begin{thm}\label{fully_converge}
Suppose that \(v\) is the exact solution of \eqref{eq:v-problem} with $v$, ${}_0^C D_t^{\alpha} v\in L^\infty(0,T;H^m(\Omega))$ and $\|\partial_t^\ell v\|_{m} \le C\bigl(1+t^{\delta-\ell}\bigr)\;(\alpha\le \delta<2,  \ell=0,1,2)$. Assume that the conditions of
Lemma~\ref{coe2} hold and that \(r\ge 4\) for graded mesh. Then the numerical solution
\(\{v_N^n\}_{n=1}^M\) generated by \eqref{eq:fast_scheme_weak}, with the
initial value \(v_N^0=\pi_N^{1,0}v^0\), satisfies
\begin{equation*}
\begin{split}
\|v(t_n)-v_N^n\|_1\le C_3(\Theta_\delta(M)N^{-m}+N^{1-m}+M^{-2}+\varepsilon),\quad 1\le n \le M,
\end{split}
\end{equation*}
where $C_3$ is a constant independent of $\varepsilon$, $M$ and $N$,
\begin{equation*}
\Theta_\delta(M)=\begin{cases}1,&\qquad \delta>\frac{1}{2},\\[3pt]
(1+\ln M)^{1/2},&\qquad \delta=\frac{1}{2},\\[4pt]
M^{{r}(1-2\delta)/2},&\qquad 0<\delta<\frac{1}{2}.\end{cases}
\end{equation*}
\end{thm}
\begin{proof}
We denote
\begin{equation}\label{notation1}
\begin{split}
&e_N^n:=v_N^n-\pi_N^{1,0}v(t_n),
\quad\qquad
e_N^{n-\frac12}:=v_N^{n-\frac12}-\tfrac{1}{2}\pi_N^{1,0}(v(t_n)+v(t_{n-1})),\\[4pt]
&\partial_t^\alpha v(t_n)
:=\frac{1}{\tau_n}\int_{t_{n-1}}^{t_n}{}_0^C D_t^\alpha v(x,t)\,dt,
\quad\quad
v(t_{n-\frac12})
:=\frac{1}{\tau_n}\int_{t_{n-1}}^{t_n}v(x,t)\,dt.\\[2pt]
\end{split}
\end{equation}
 Based on \eqref{eq:fast_scheme_weak} and \eqref{notation1}, 
we derive for $\forall \phi_N\in\mathbb{Q}_N^0(\Omega)$,
\begin{equation}\label{eq:fully_error_eq}
\begin{split}
    &(\nabla_\tau e_N^n,\phi_N)_N
+(\bar{\partial}_{\tau,\varepsilon}^\alpha e_N^n,\phi_N)_N
+(\nabla e_N^{n-\frac12},\nabla\phi_N)_N
-(\mu e_N^{n-\frac12},\phi_N)_N \\[5pt]
=\;&\varepsilon_1^n(\phi_N)+\varepsilon_2^n(\phi_N)+\varepsilon_3^n(\phi_N),
\end{split}
\end{equation}
where\vspace{-0.15cm}
\begin{equation}\label{ee3}
\varepsilon_1^n(\phi_N)
:=
\bigl(\pi_N^{1,0}R_2^n,\phi_N\bigr)_N,    \qquad  \varepsilon_2^n(\phi_N)
:=
\bigl(\bar{\partial}_{\tau,\varepsilon}^\alpha v(t_n)
-\bar{\partial}_{\tau,\varepsilon}^\alpha \pi_N^{1,0}v(t_n),\phi_N\bigr)_N,
\end{equation}
\begin{equation}\label{ee2}
\begin{split}
    \varepsilon_3^n(\phi_N)
:=&
-\bigl(\bar{\partial}_{\tau,\varepsilon}^\alpha v(t_n),\phi_N\bigr)_N
-\bigl(\nabla_\tau \pi_N^{1,0}v(t_n),\phi_N\bigr)_N\\[5pt]
&-\bigl(\nabla \pi_N^{1,0}v(t_{n-\frac12}),\nabla \phi_N\bigr)_N
+\bigl(\mu\pi_N^{1,0}v(t_{n-\frac12}),\phi_N\bigr)_N+(\bar f^{n-\frac{1}{2}},\phi_N).\\[3pt]
\end{split} 
\end{equation}
Applying \eqref{eq:spec3}-\eqref{eq:spec2} to \eqref{ee3} yields
\begin{equation}\label{final3}
|\varepsilon_1^n(\phi_N)|
\le (cN^{-m}\|R_2^n\|_m+\|R_2^n\|)\,\|\phi_N\|,\vspace{0.1cm}
\end{equation}
where $R_2^n$ is defined in \eqref{eq:error2}. Based on \eqref{eq:error1} and \eqref{eq:spec2}, we derive
\vspace{0.05cm}
\begin{equation}\label{finale1_1}
\begin{split}
|\varepsilon_2^n(\phi_N)|
=\;&
|
((I_d-\pi_N^{1,0})(\frac{1}{\tau_n}\int_{t_{n-1}}^{t_n} {}_0^C D_t^{\alpha} v\, dt-R_1^n),\phi_N)_N
|\\[5pt]
\le \;&  \|(I_d-\pi_N^{1,0})(\frac{1}{\tau_n}\int_{t_{n-1}}^{t_n} {}_0^C D_t^{\alpha} v\, dt-R_1^n)\|\|\phi_N\|\\[5pt]
&+
c N^{-m}
\|(I_d-\pi_N^{1,0})(\frac{1}{\tau_n}\int_{t_{n-1}}^{t_n} {}_0^C D_t^{\alpha} v\, dt-R_1^n
)\|_{m}\|\phi_N\|,\\[3pt]
\end{split}
\end{equation}
where $I_d$ denotes the identity operator and $R_1^n$ is defined in \eqref{eq:error1}. With \eqref{eq:spec3}, we further derive from \eqref{finale1_1}
\begin{equation}\label{finale1}
|\varepsilon_2^n(\phi_N)|
\le cN^{-m}(\|{}_0^C D_t^{\alpha} v\|_{L^\infty(0,T;H^m(\Omega))}+\|R_1^n\|_m)\|\phi_N\|.
\end{equation}
According to \eqref{norm_n}, we have for $\forall \phi_N\in\mathbb{Q}_N^0(\Omega)$
\vspace{0.1cm}
\begin{equation}\label{small}
\bigl(\nabla \pi_N^{1,0}v(t_{n-\frac12}),\nabla \phi_N\bigr)_N
=\bigl(\nabla \pi_N^{1,0}v(t_{n-\frac12}),\nabla \phi_N\bigr)
=
\bigl(\nabla v(t_{n-\frac12}),\nabla \phi_N\bigr)
.\vspace{0.1cm}
\end{equation}
Substituting \eqref{small} to \eqref{ee2} yields
\vspace{0.1cm}
\begin{equation}\label{26051811}
\begin{split}
\varepsilon_3^n(\phi_N)
=\;&
-\bigl(\bar{\partial}_{\tau,\varepsilon}^\alpha v(t_n),\phi_N\bigr)_N
-\bigl(\nabla_\tau \pi_N^{1,0}v(t_n),\phi_N\bigr)_N\\[6pt]
\;&-\bigl(\nabla v(t_{n-\frac12}),\nabla \phi_N\bigr)
+\bigl(\mu\pi_N^{1,0}v(t_{n-\frac12}),\phi_N\bigr)_N+(\bar f^{n-\frac{1}{2}},\phi_N).\\[3pt]
\end{split}
\end{equation}
With \eqref{average}, \eqref{26051811} becomes
\begin{equation}\label{2605181}
\begin{split}
\varepsilon_3^n(\phi_N)
=\;&-\bigl(\bar{\partial}_{\tau,\varepsilon}^\alpha v(t_n),\phi_N\bigr)_N-\bigl(\nabla_\tau \pi_N^{1,0}v(t_n),\phi_N\bigr)_N+\bigl(\mu\pi_N^{1,0}v(t_{n-\frac12}),\phi_N\bigr)_N\\[6pt]
\;&
+\bigl(\frac{1}{\tau_n}\int_{t_{n-1}}^{t_n} {}_0^C D_t^{\alpha} v\, dt,\phi_N\bigr) +\bigl(\nabla_\tau v(t_n),\phi_N\bigr)-\bigl(\mu v(t_{n-\frac12}),\phi_N\bigr).
\end{split}
\end{equation}
By adding and subtracting
\((\pi_N^{1,0}\nabla_\tau v(t_n),\phi_N)\), we derive from \eqref{eq:spec3}-\eqref{eq:spec2}
\begin{equation}\label{eps22_split}
\begin{aligned}
\;&\bigl|\bigl(\nabla_\tau v(t_n),\phi_N\bigr)-\bigl(\nabla_\tau \pi_N^{1,0}v(t_n),\phi_N\bigr)_N\bigl|\\[6pt]
=\;&
\bigl(
(I_d-\pi_N^{1,0})\nabla_\tau v(t_n),\phi_N
\bigl)+
\bigl(
\pi_N^{1,0}\nabla_\tau v(t_n),\phi_N
\bigl)
-
\bigl(
\pi_N^{1,0}\nabla_\tau v(t_n),\phi_N
\bigl)_N \\[6pt]
\le \;&
C N^{-m}
\|\nabla_\tau v(t_n)\|_{m}
\|\phi_N\|
\le
\,C N^{-m}
(
\frac{1}{\tau_n}
\int_{t_{n-1}}^{t_n}
\|\partial_t v(t)\|_{m}\,dt
)
\|\phi_N\|\\
\le \;& \frac{C}{\tau_n}N^{-m}({t_n^{\delta}-t_{n-1}^\delta})\|\phi_N\|.
\end{aligned}
\end{equation}
Similar to \eqref{finale1}, substituting \eqref{eps22_split} into \eqref{2605181} and applying \eqref{eq:spec3}-\eqref{eq:spec2} yields
\begin{equation}\label{finale2}
\begin{split}
|\varepsilon_3^n(\phi_N)|
\le \;&
CN^{-m}\Bigl(
\|{}_0^C D_t^{\alpha} v\|_{L^\infty(0,T;H^m(\Omega))}
+ \|v\|_{L^\infty(0,T;H^m(\Omega))}\\[4pt]
\;&+\tfrac{1}{\tau_n}({t_n^{\delta}-t_{n-1}^\delta})+\|R_1^n\|_m
\Bigr)\|\phi_N\|+\|R_1^n\|\|\phi_N\|.\\[4pt]
\end{split}
\end{equation}
Combining \eqref{final3}, \eqref{finale1} and \eqref{finale2} yields
\vspace{0.15cm}
\begin{equation}\label{eq:rhs_estimate}
\begin{split}
    &|\varepsilon_1^n(\phi_N)|
+|\varepsilon_2^n(\phi_N)|
+|\varepsilon_3^n(\phi_N)|\\[5pt]
\le \;&
cN^{-m}\Bigl(
\|{}_0^C D_t^{\alpha} v\|_{L^\infty(0,T;H^m(\Omega))}+\|R_1^n\|_m+\|R_2^n\|_m
\\[4pt]
\;& +\tfrac{1}{\tau_n}({t_n^{\delta}-t_{n-1}^\delta})
\Bigr)\|\phi_N\|+(\|R_1^n\|+\|R_2^n\|)\|\phi_N\|.\\[4pt]
\end{split}
\end{equation}
Following the same argument as that of Theorem~\ref{thm:semi}, we obtain from \eqref{eq:spec1}, \eqref{eq:fully_error_eq} and \eqref{eq:rhs_estimate}
\vspace{0.1cm}
\begin{equation}\label{eq:final_err}
\begin{split}
&\sum\limits_{k=1}^n\frac{1}{2\tau_k}{\|e_N^k-e_N^{k-1}\|^2}
+\frac{\sigma(\alpha)}{2\Gamma(3-\alpha)}
\sum\limits_{k=1}^n\frac{1}{\tau_k^\alpha}{\|e_N^k-e_N^{k-1}\|^2}
+\frac12{\|e_N^n\|_1^2}
\\[3pt]
\le \;& \sum\limits_{k=1}^n\frac{1}{2\tau_k}{\|e_N^k-e_N^{k-1}\|_N^2}
+\frac{\sigma(\alpha)}{2\Gamma(3-\alpha)}
\sum\limits_{k=1}^n\frac{1}{\tau_k^\alpha}{\|e_N^k-e_N^{k-1}\|_N^2}
+\frac12{\|\nabla e_N^n\|_N^2}
\\
\le \;&C_d
\sum\limits_{k=1}^n\tau_k
\Bigl(cN^{-m}\Bigl(
\|{}_0^C D_t^{\alpha} v\|_{L^\infty(0,T;H^m(\Omega))}
+\|R_1^k\|_m+\|R_2^k\|_m\\[2pt]
\;&+
\tfrac{1}{\tau_k}({t_k^{\delta}-t_{k-1}^\delta})
\Bigr)+\bigl(\|R_1^k\|+\|R_2^k\| \bigr)
\Bigr)^2.
\end{split}
\end{equation}
 By Lemma~\ref{lem:accumulated_R}, we have
\begin{equation}\label{202660116}
\sum\limits_{k=1}^n\tau_k\|R^k\|_m
\le
C_2(\varepsilon+M^{-2})\qquad \mathrm{and} \qquad\sum\limits_{k=1}^n\tau_k\|R^k\|_m^2
\le
C_2(\varepsilon+M^{-2})^2.
\end{equation}
Based on the graded mesh
$t_k = T (k/M )^r$, we have $\tau_k\le C M^{-r}k^{r-1}$. By the mean value theorem, we derive
\begin{equation}\label{eq:Qk_est}
\sum_{k=1}^n \tau_k \left(\tfrac{t_k^{\delta}-t_{k-1}^\delta}{\tau_k}\right)^2
\le
C
\sum_{k=1}^n
M^{-r}k^{r-1}
\left(M^{-r}k^r\right)^{2\delta-2}
\le C \Theta_\delta(M)^2.\vspace{0.1cm}
\end{equation}
Substituting \eqref{202660116} and \eqref{eq:Qk_est} into \eqref{eq:final_err}, we obtain
\vspace{0.15cm}
\begin{equation}\label{eee1}
    \|e_N^n\|_1
\le CN^{-m}(\|{}_0^C D_t^{\alpha} v\|_{L^\infty(0,T;H^m(\Omega))}+\Theta_\delta(M))+C(\varepsilon+M^{-2}),
 \vspace{0.2cm}
\end{equation}
 where $C$ is a constant independent of $\varepsilon$, $M$ and $N$. Taking \(\varphi=v(t_n)\) in \eqref{eq:spec3} gives
 \begin{equation}\label{eee2}
     \|v(t_n)-\pi_N^{1,0}v(t_n)\|_1\le cN^{1-m}\|v(t_n)\|_m\le cN^{1-m}\|v\|_{L^{\infty}(0,T;H^m(\Omega))}. \vspace{0.1cm}
 \end{equation} 
 Applying the triangle inequality together with \eqref{eee1} and \eqref{eee2}, we obtain
 \vspace{0.15cm}
\begin{equation*}
\begin{split}
\|v(t_n)-v_N^n\|_1\le\;
&\|v(t_n)-\pi_N^{1,0}v(t_n)\|_1+\|e_N^n\|_1\\[6pt]
\le\;
&C_3N^{-m}(\|{}_0^C D_t^{\alpha} v\|_{L^\infty(0,T;H^m(\Omega))}+\Theta_\delta(M))\\[6pt]
&+C_3N^{1-m}\|v\|_{L^\infty(0,T;H^m(\Omega))}+C_3(\varepsilon+M^{-2}),\\[3pt]
\end{split}
\end{equation*}
where $C_3$ is a constant independent of $\varepsilon$, $M$ and $N$. This completes the proof.
\end{proof}
\begin{remark}
Theorem~\ref{fully_converge} shows that the full discretization scheme \eqref{eq:fast_scheme_weak} achieves second-order temporal convergence in the $H^1$-norm, which is higher than that of the method in \cite{cao2020}. Numerical experiments in the next section confirm this result.
\end{remark}

\section{Numerical Experiments}
\label{sec:numerical}
In this section, we test three examples to verify the validity of Theorems~\ref{thm:semi} and ~\ref{fully_converge}, where we choose $r = 4$, $\delta=\alpha$ and $\varepsilon=O(M^{-2})$. Furthermore, based on our discretization scheme \eqref{eq:fast_scheme_weak}, we discuss the effects of the fractional parameters $\alpha$ and $\lambda$ on the solution of \eqref{Equation}. The computations are carried out on a server: Intel(R) Core(TM) i5 CPU at 2.40 GHz and 8 GB of RAM via MATLAB (version R2023b).\vspace{0.2cm}

{\bf Example 1.} As seen in \cite{cao2020,chen2018,zhang2013}, we consider \eqref{Equation} with
\(T=2\), \(\Omega=(0,1)\), \(\phi(x)=x^2(1-x)^2\), and \vspace{0.05cm} 
\begin{equation}
\begin{split}
    f(x,t) =\;& e^{-\lambda t} \bigl\{ ( -\lambda(t^\delta+t^2+1) +\delta t^{\delta-1} +2t +\tfrac{\Gamma(\delta+1)}{\Gamma(\delta-\alpha+1)}t^{\delta-\alpha})   \\[4pt] 
    &x^2(1-x)^2- (t^\delta+t^2+1)  (-4x^3+18x^2-14x+2)  \bigr\}.\\[3pt]
\end{split}
\end{equation}
The exact solution is $u(x,t)=e^{-\lambda t}(t^\delta+t^2+1)x^2(1-x)^2$. For an appropriate choice of $\alpha$, Fig.~\ref{fig:time_plot} shows that our scheme achieves second-order temporal convergence in the $H^1$-norm. This is consistent with the theoretical results in Theorems~\ref{thm:semi} and~\ref{fully_converge} and outperforms the convergence rate reported in \cite{cao2020}. Fig.~\ref{fig:1d_space} plots the errors versus the polynomial degree $N$ with $M=30000$. As expected for spectral methods, the error curves exhibit exponential convergence with respect to $N$, confirming the spatial error estimate in Theorem~\ref{fully_converge}.
  
\vspace{-0.1cm}
\begin{figure}[htbp]
	\begin{minipage}[t]{0.33\linewidth}
  \centering
  \includegraphics[width=\textwidth]{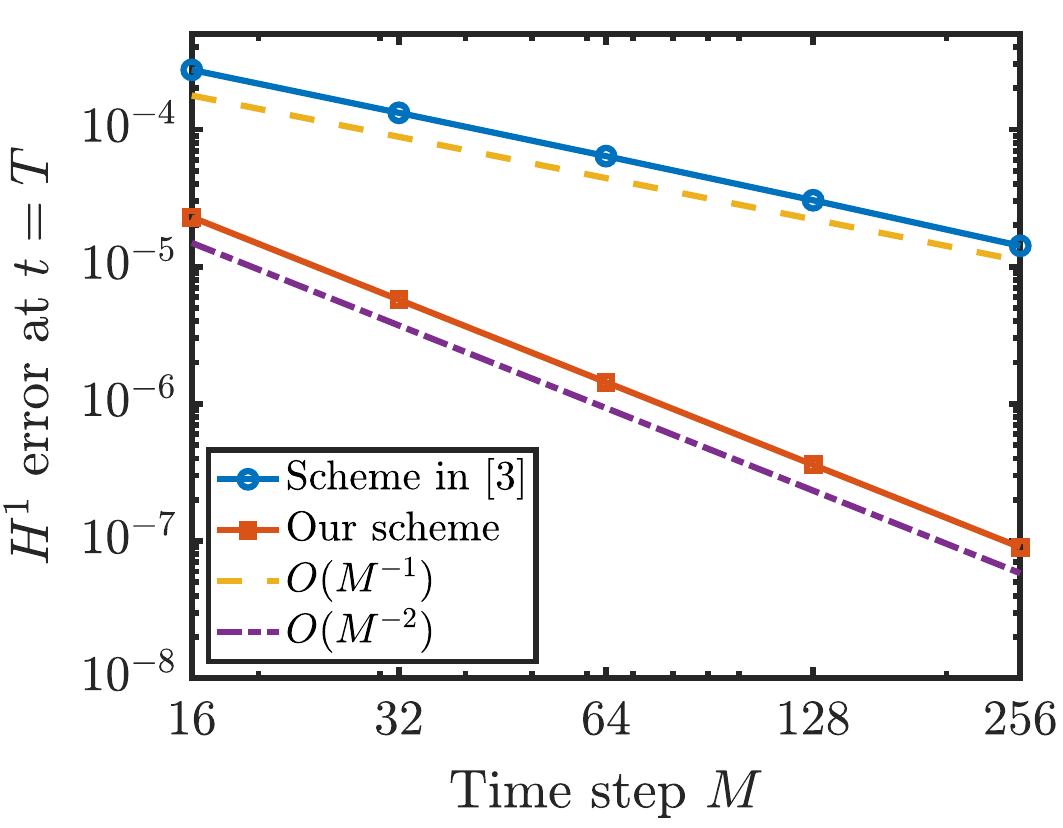}
  \centerline{\quad\;\;(a) $\alpha=0.25$}
	\end{minipage}%
	\begin{minipage}[t]{0.33\linewidth}
   \centering
   \includegraphics[width=\textwidth]{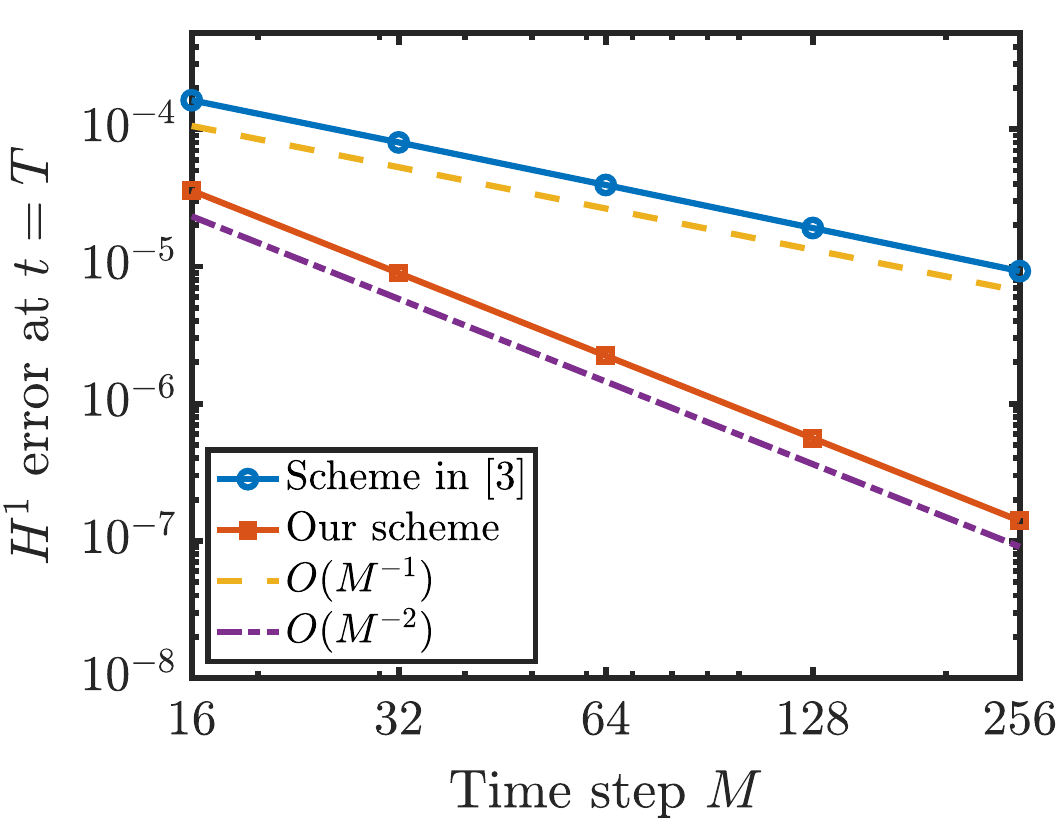}
   \centerline{\quad\;\;(b) $\alpha=0.5$}
	\end{minipage}
    \begin{minipage}[t]{0.33\linewidth}
   \centering
   \includegraphics[width=\textwidth]{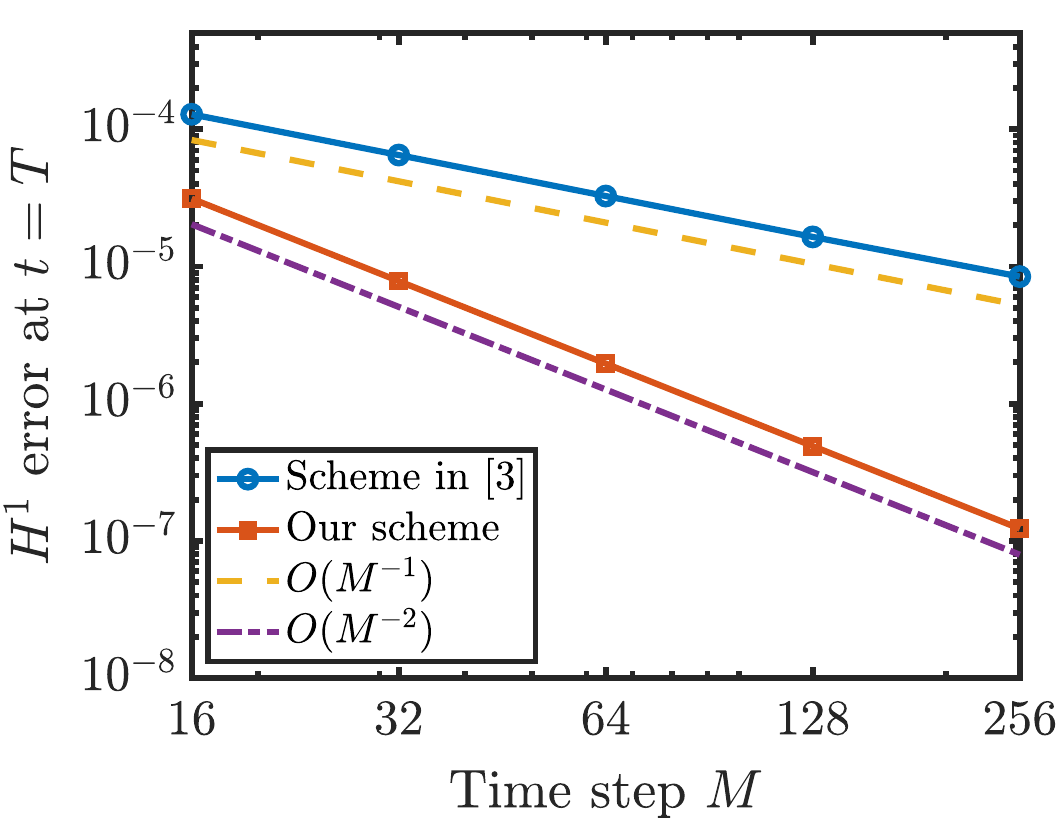}
   \centerline{\quad\;\;(c) $\alpha=0.75$}
	\end{minipage}
    \vspace{-0.3cm}
\caption{A comparison between our scheme \eqref{eq:fast_scheme_weak} and the scheme in \cite{cao2020}.}\label{fig:time_plot}\vspace{-0.5cm}\vspace{-0.3cm}
\end{figure}  

\vspace{-0.3cm}
\begin{figure}[htbp]
    \centering
\includegraphics[width=0.49\textwidth]{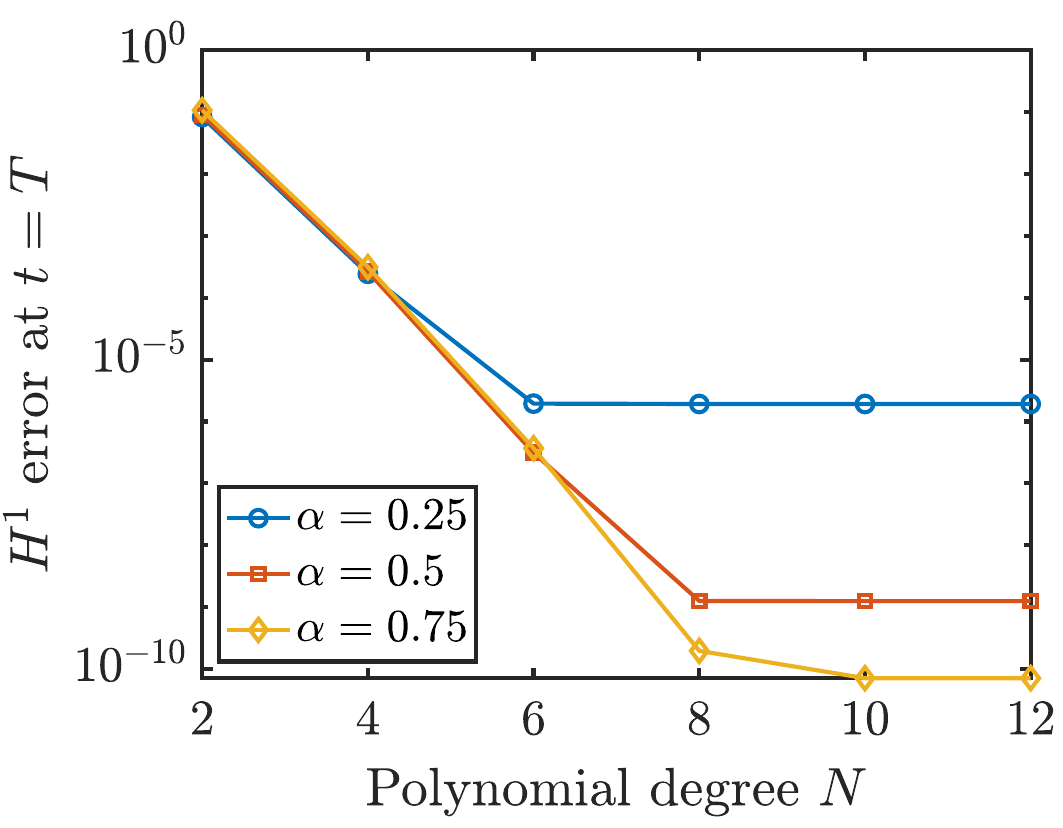}
    \vspace{-0.3cm}
    \caption{Errors vs. polynomial degree $N$ for our scheme \eqref{eq:fast_scheme_weak}.}
    \label{fig:1d_space} \vspace{-0.3cm}
\end{figure}

When the case $\alpha = 1$ is considered, we have from \eqref{Equation} 
\begin{equation}\label{integer}
    2\partial_tu+\lambda u=\Delta u - \sum\limits_{j=1}^d \partial_{x_j}u + f(\boldsymbol{x},t).
\end{equation}
due to $\lim
\limits_{\alpha\to 1^-}{}_0^C D_t^{\alpha,\lambda} u=\partial_tu+\lambda u$. The full discretization scheme \eqref{eq:fast_scheme_weak} becomes the Crank--Nicolson-type discretization. Similar to the case $\alpha \in (0, 1)$, we obtain the same convergence results (see Fig.~\ref{fig:alpha1_1}). In addition, based on our scheme \eqref{eq:fast_scheme_weak}, we further investigate how the solution of \eqref{Equation} varies with the parameters $\alpha$ and $\lambda$. Fig.~\ref{fig:vary1}(a) presents numerical solutions for varying $\lambda$. As $\lambda$ increases, the solutions decay more rapidly in time, reflecting the exponential tempering effect introduced by the Caputo tempered fractional derivative ($\lambda>0$). This behavior agrees with the description in \cite{xia2013}, where the long-time tail transitions from a power-law to an exponentially tempered form. In Fig.~\ref{fig:vary1}(b), we set $\lambda=1$ to examine the effect of $\alpha$. Compared with the case \(\alpha=1\), smaller values of \(\alpha\) lead to a stronger short-time response and a faster long-time decay. Moreover, the numerical solutions tend asymptotically to the solution of the integer-order equation \eqref{integer} as $\alpha \to 1$.

\vspace{-0.2cm}
\begin{figure}[htbp]
    \centering
\includegraphics[width=0.45\textwidth]{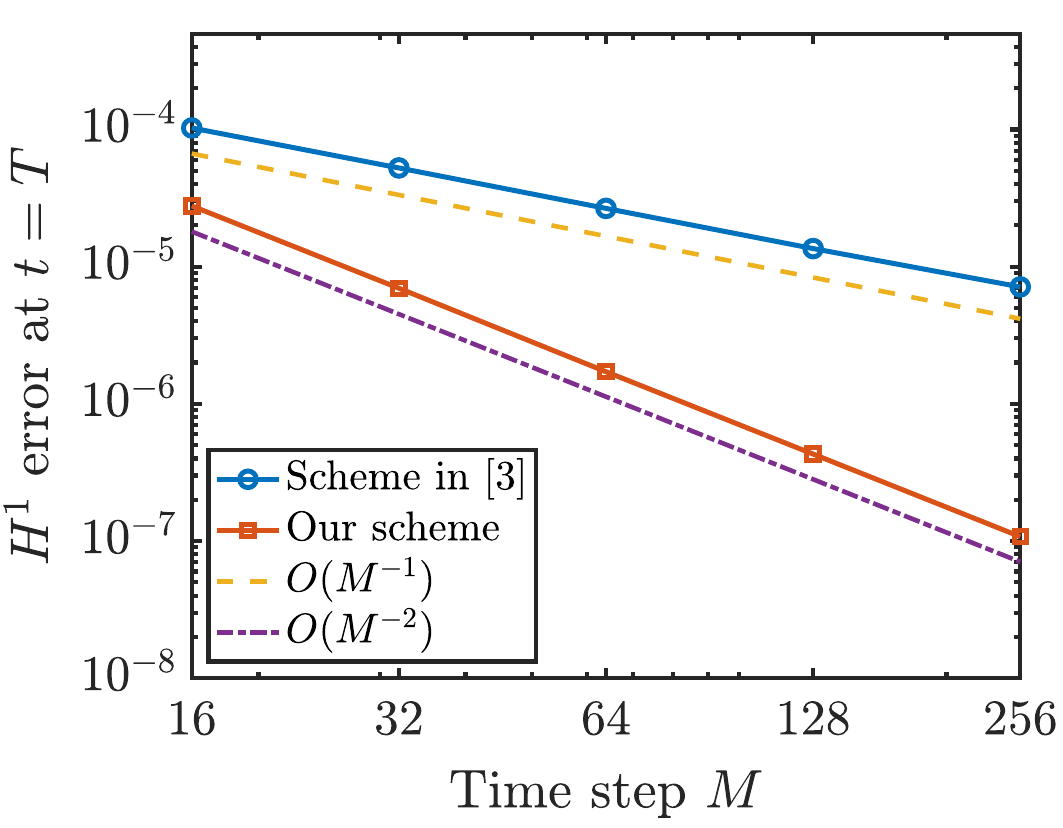}
    \hspace{0.1cm} 
\includegraphics[width=0.45\textwidth]{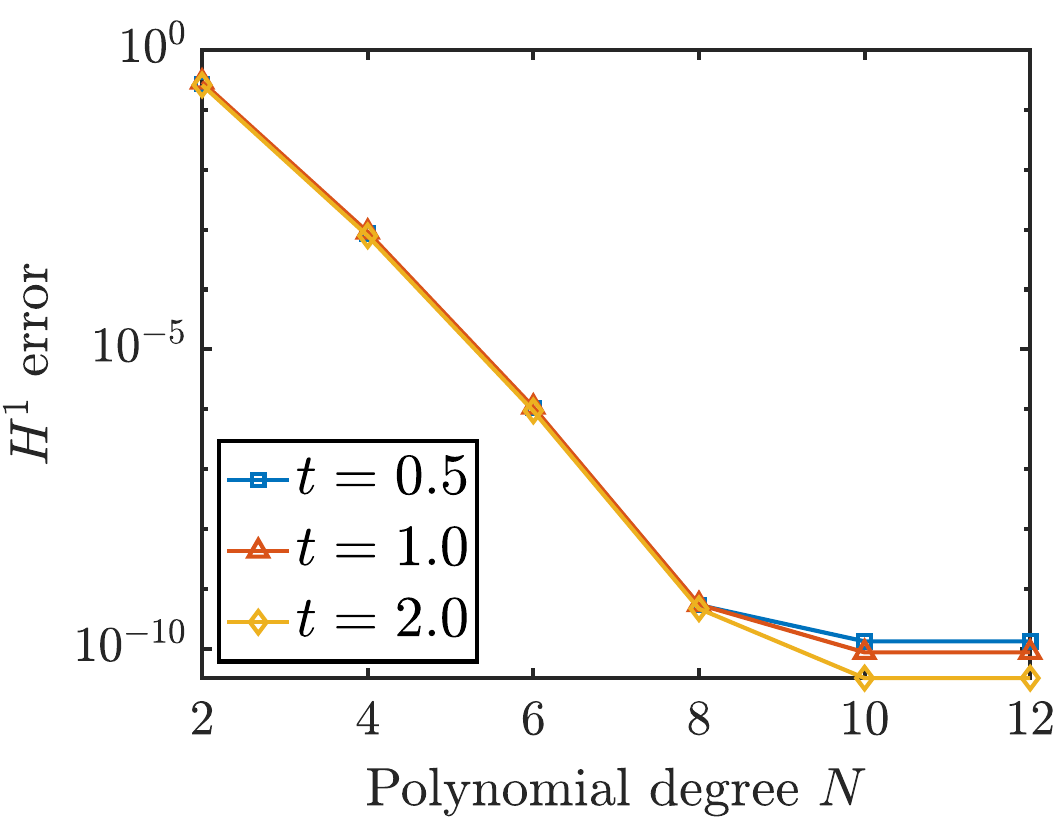}
    \vspace{-0.2cm}
    \caption{Left: Comparison of temporal convergence between our scheme \eqref{eq:fast_scheme_weak} and the scheme in \cite{cao2020}. Right: Errors vs. polynomial degree $N$ for our scheme \eqref{eq:fast_scheme_weak}.}
    \label{fig:alpha1_1}
    \vspace{-0.5cm}
\end{figure}

\vspace{-0.5cm}
\begin{figure}[htbp]
	\begin{minipage}[t]{0.45\linewidth}
  \centering
  \includegraphics[width=\textwidth]{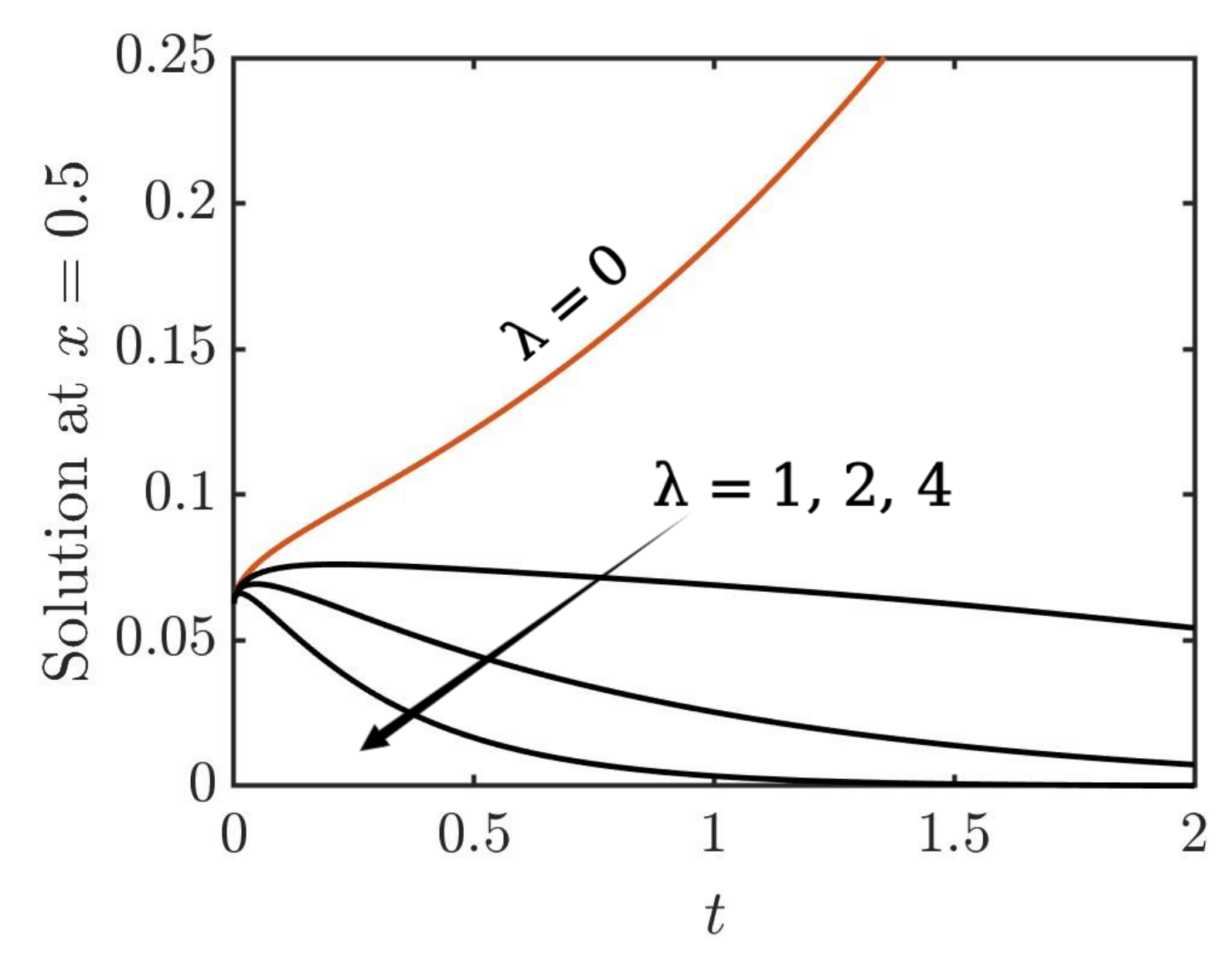}
  \centerline{\qquad(a) $\alpha=0.5$}
	\end{minipage}%
    \hspace{0.3cm} 
	\begin{minipage}[t]{0.45\linewidth}
   \centering
   \includegraphics[width=\textwidth]{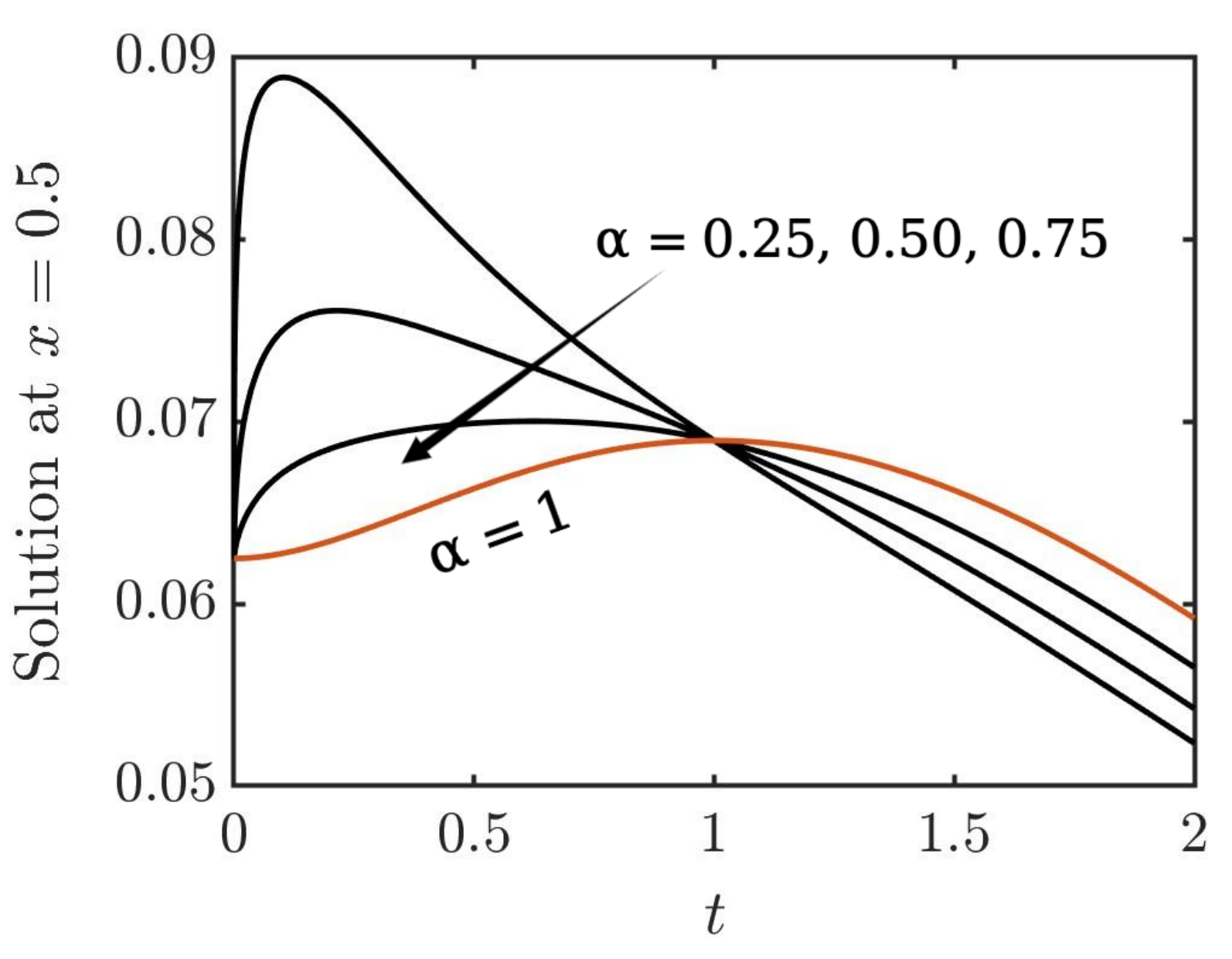}
   \centerline{\qquad(b) $\lambda=1$}
	\end{minipage}
    \vspace{-0.15cm}
\caption{Numerical solutions of \eqref{Equation} with varying $\lambda$ and $\alpha$ at $x=0.5$.}\label{fig:vary1}\vspace{-0.5cm}
\end{figure}

{\bf Example 2.} For \eqref{Equation}, we consider  
\(\Omega=(0,1)\times(0,1)\), \(T=2\), 
\begin{equation*}
\phi(x,y) = 16(x-\tfrac12)^2 (y-\tfrac12)^2 \sin(\pi x)\sin(\pi y) e^{-20((x-\frac12)^2+(y-\frac12)^2)},
\end{equation*}
\begin{equation*}
\begin{split}
f(x,y,t)
=
e^{-\lambda t}
\bigl\{
&
\left(
-t^{\alpha-1}E_{\alpha,\alpha}(-t^\alpha)
-(\lambda+1)E_{\alpha,1}(-t^\alpha)
\right)\phi(x,y)
\\
&-
E_{\alpha,1}(-t^\alpha)
\left(
\phi_{xx}(x,y)+\phi_{yy}(x,y)-\phi_x(x,y)-\phi_y(x,y)
\right)
\bigr\}.
\end{split}
\end{equation*}
The exact solution is
\begin{equation*}
u(x,y,t)=e^{-\lambda t}E_{\alpha,1}(-t^\alpha)\phi(x,y),\vspace{0.1cm}
\end{equation*}
where $E_{\alpha,1}(-t^\alpha)
:=
\sum_{k=0}^{\infty}
\frac{(-t^\alpha)^k}{\Gamma(\alpha k+1)}$. Fig.~\ref{fig:time_plot2} presents the temporal convergence results, demonstrating the second-order convergence rate of our scheme. These results are in good agreement with Theorems~\ref{thm:semi} and~\ref{fully_converge}. Moreover, our scheme achieves significantly higher accuracy than the method in \cite{cao2020}. The second-order temporal convergence in the $H^1$-norm is also retained when \eqref{Equation} reduces to the integer-order case ($\alpha=1$). In Fig.~\ref{fig:2d_space1}, the $H^1$-norm errors are plotted against the polynomial degree $N$, with $M=50$. Clearly, the spectral accuracy for the spectral collocation discretization in space is achieved, confirming the convergence estimate stated in Theorem~\ref{fully_converge}. 
 
\vspace{-0.1cm}
\begin{figure}[htbp]
	\begin{minipage}[t]{0.45\linewidth}
  \centering
  \includegraphics[width=\textwidth]{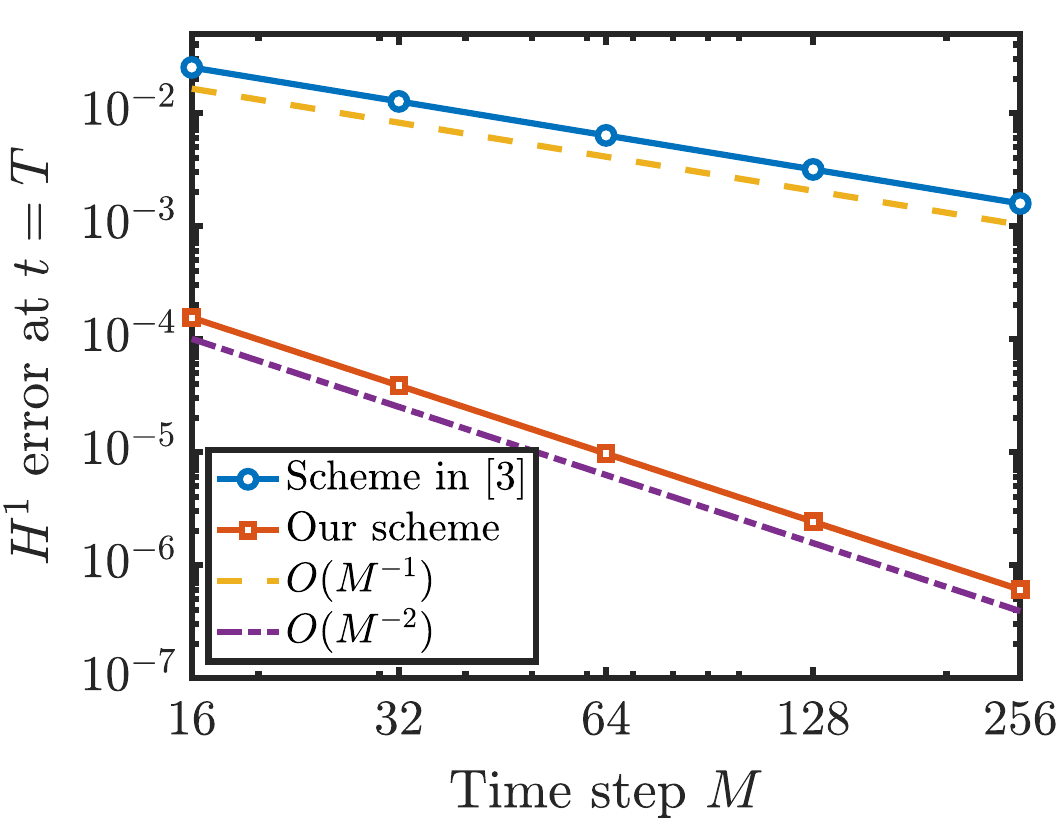}
  \centerline{\qquad\;(a) $\alpha=0.25$}
	\end{minipage}%
    \hspace{0.32cm} 
	\begin{minipage}[t]{0.45\linewidth}
   \centering
   \includegraphics[width=\textwidth]{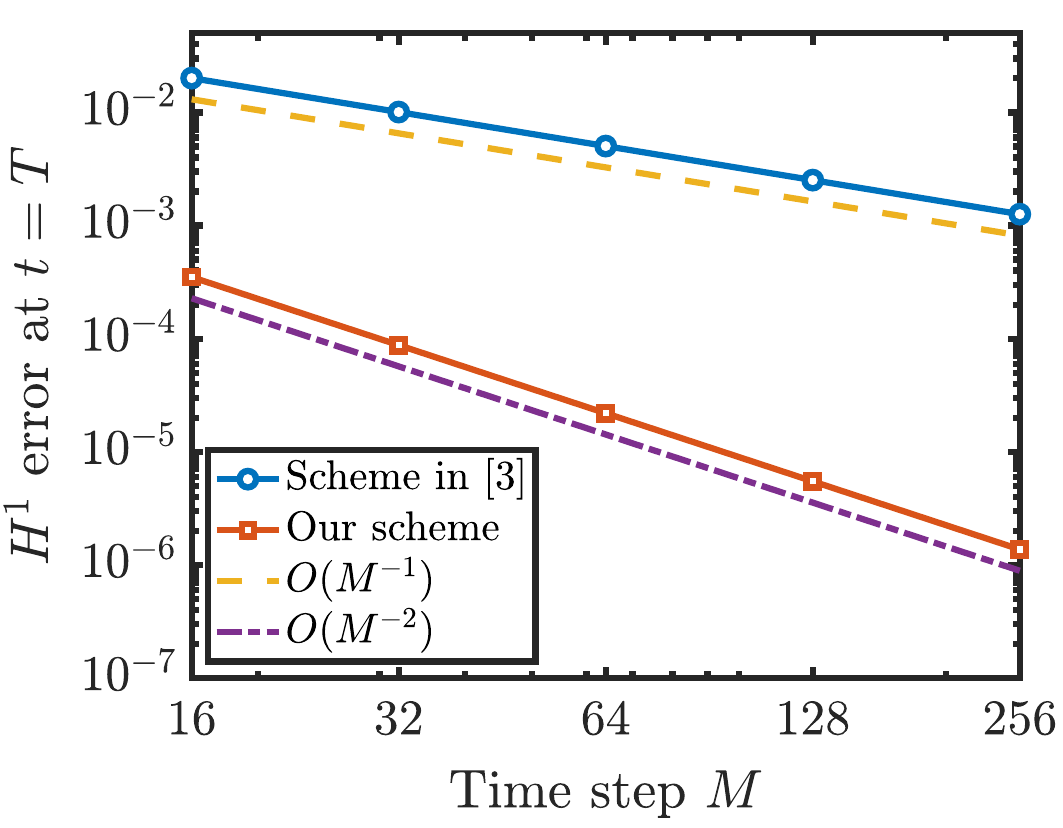}
   \centerline{\qquad\;(b) $\alpha=0.5$}
	\end{minipage}
    \begin{minipage}[t]{0.45\linewidth}
   \centering
   \includegraphics[width=\textwidth]{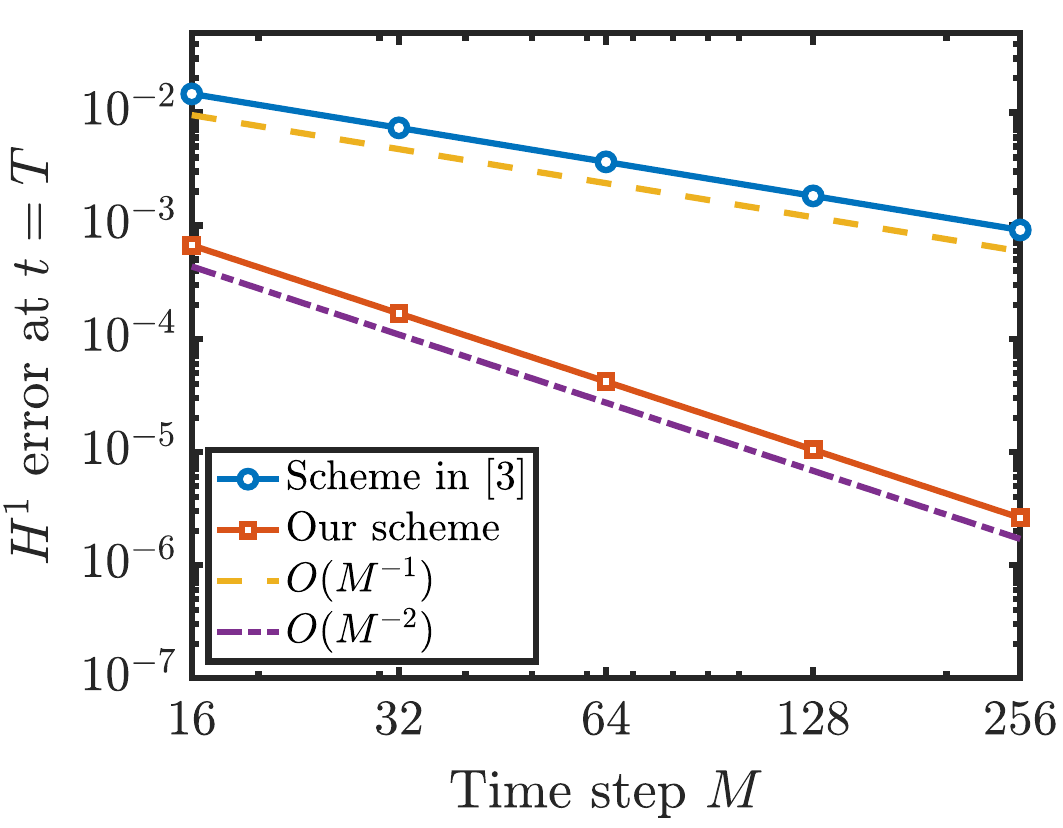}
   \centerline{\qquad\;(c) $\alpha=0.75$}
	\end{minipage}
    \hspace{0.2cm} 
    \begin{minipage}[t]{0.45\linewidth}
   \centering
   \includegraphics[width=\textwidth]{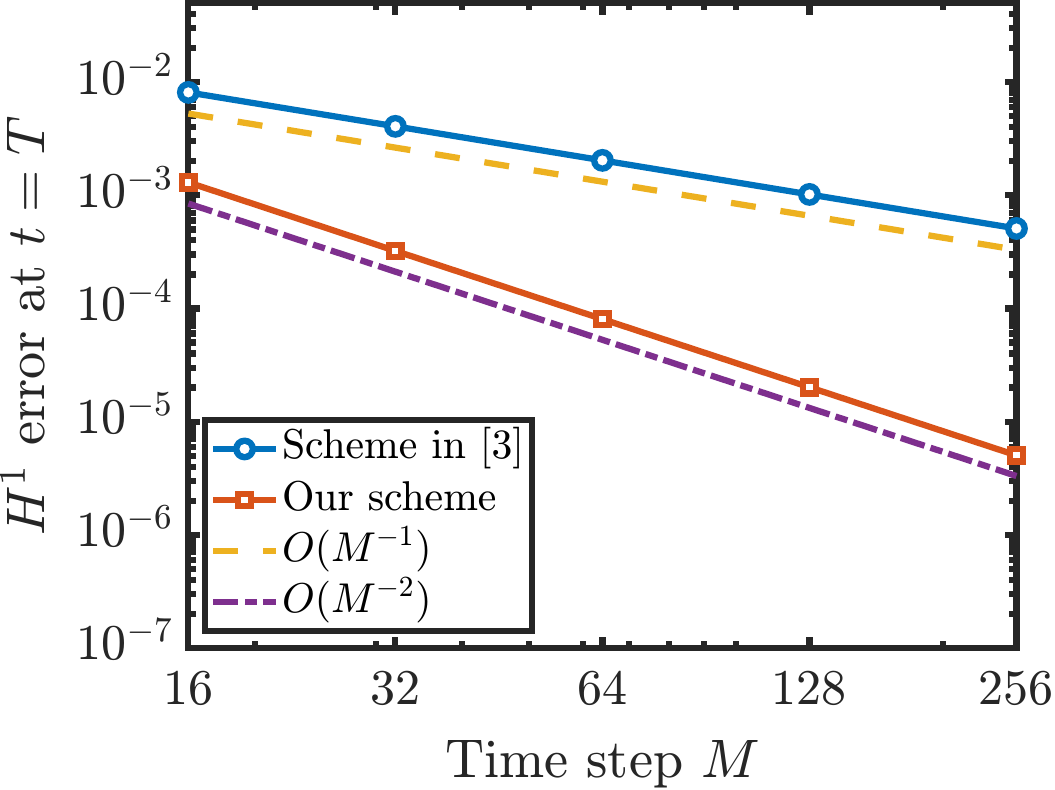}
   \centerline{\qquad\;(d) $\alpha=1.0$}
	\end{minipage}
    \vspace{-0.1cm}
\caption{A comparison between our scheme \eqref{eq:fast_scheme_weak} and the scheme in \cite{cao2020} for {\bf Example 2}.}\label{fig:time_plot2}\vspace{-0.5cm}
\end{figure}  

\begin{figure}[htbp]
    \centering
\includegraphics[width=0.49\textwidth]{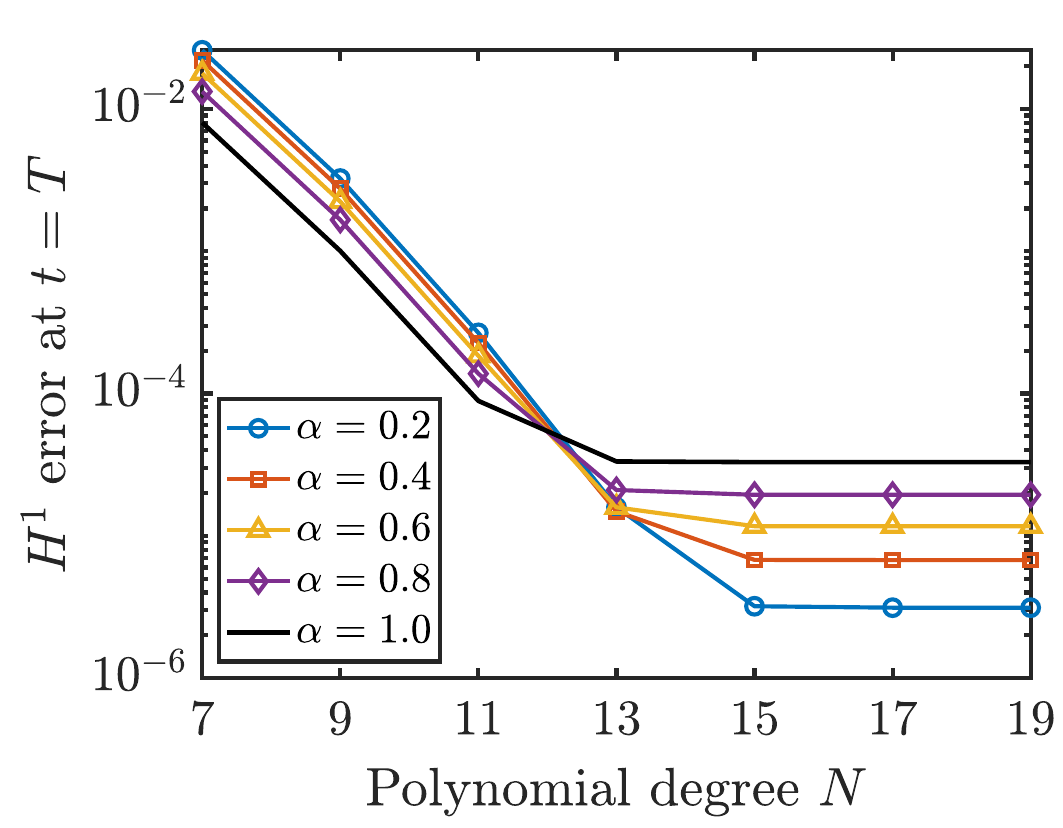}
    \vspace{-0.2cm}
    \caption{ Errors vs. polynomial degree $N$ for our scheme \eqref{eq:fast_scheme_weak} for {\bf Example 2}.}\vspace{-0.3cm}
    \label{fig:2d_space1}
\end{figure}

\begin{figure}[htbp]
    \centering
\includegraphics[width=0.9\textwidth]{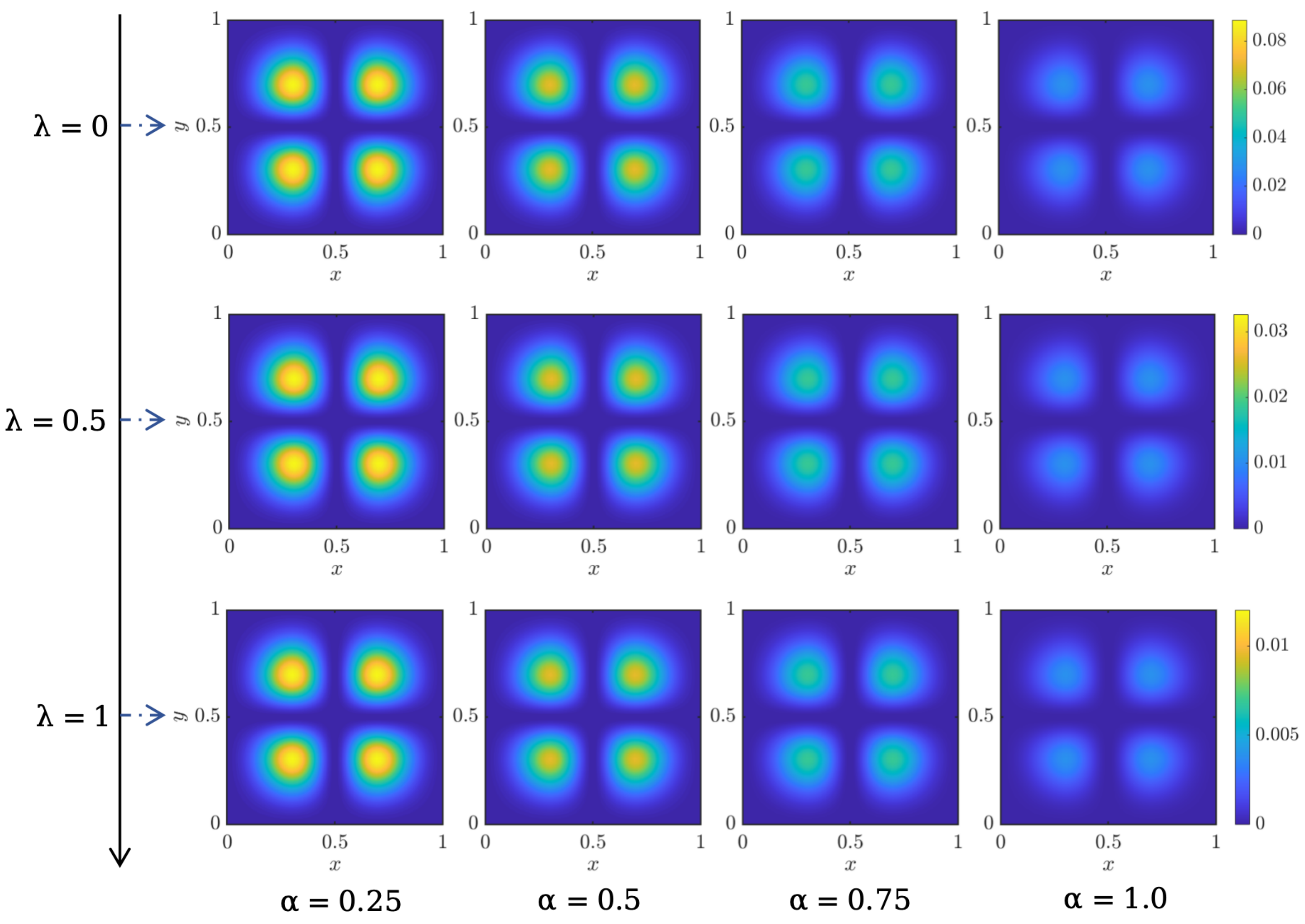}
 \vspace{-0.2cm}
\caption{Numerical solutions of \eqref{Equation} with varying $\lambda$ and $\alpha$ for {\bf Example 2}.}  \vspace{-0.5cm}
\label{fig:2d_lambda1}
\end{figure}

In Fig.~\ref{fig:2d_lambda1}, the effects of the fractional parameters $\alpha$ and $\lambda$ on the solution of \eqref{Equation} are presented. For a vertical comparison, the solution with \(\lambda=0\) has a larger magnitude than those with \(\lambda=0.5\) and \(\lambda=1\), indicating that the solution decays more slowly when \(\lambda \to 0\), while a larger \(\lambda\) leads to stronger temporal decay. These numerical results demonstrate that the fractional parameter \(\lambda\) plays an important role in controlling the long-time decay behavior of the solution.
For a horizontal comparison, the numerical solutions gradually approach the corresponding
integer-order limiting solution when \(\alpha \to 1\). When $\lambda = 1$, Fig.~\ref{fig:2d_vary_alpha} shows the effects of the fractional parameter $\alpha$ on the solution of \eqref{Equation}. Obviously, for \(0<\alpha<1\), \(E_{\alpha,1}(-t^\alpha)\) shows a faster initial
decrease than the exponential function, but decays more slowly for large
\(t\). Consequently, the fractional-order solutions with \(\alpha<1\) may
be smaller than the integer-order solution with \(\alpha=1\) at early times,
whereas they become larger at later times. This behavior reflects the memory
effect and long-time tail of fractional relaxation, which is consistent with
the phenomenon reported in \cite{feng2022}.

\begin{figure}[htbp]
    \centering
\includegraphics[width=0.9\textwidth]{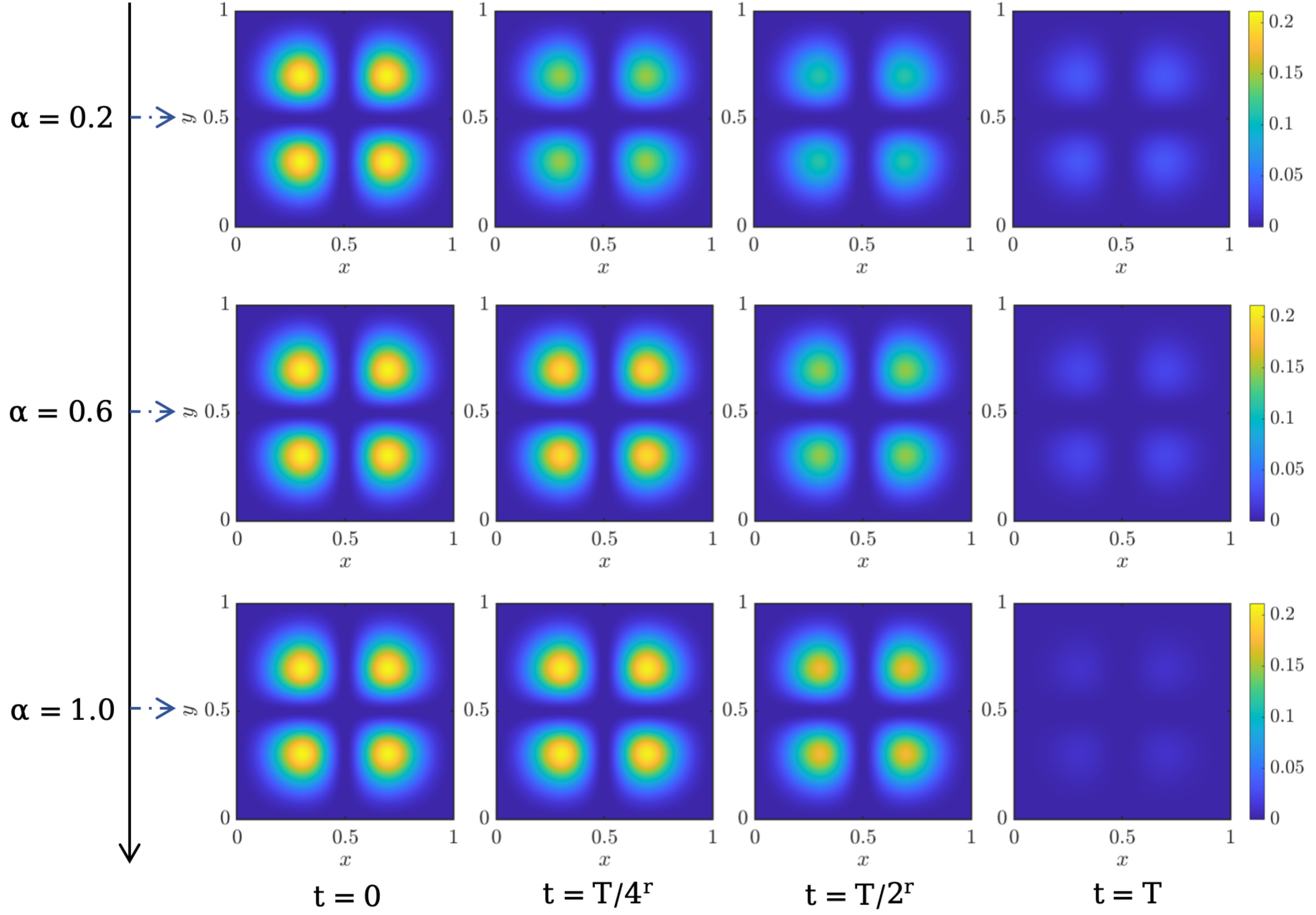}
\vspace{-0.2cm}
    \caption{Numerical solutions of \eqref{Equation} with varying $\alpha$ for {\bf Example 2}.}\vspace{-0.2cm}
    \label{fig:2d_vary_alpha}
\end{figure}
 
\vspace{0.2cm}
{\bf Example 3.} We consider $\Omega=(0,1)\times(0,1)$ and $T=2$ for \eqref{Equation} with
\begin{equation*}
\begin{split}
\phi(x,y)
=
\;&2\sin(\pi x)\sin(\pi y)
\Big\{
0.6e^{-120\left((x-\frac12)^2+(y-\frac12)^2\right)}\\
&+\sum_{m=0}^{5}
e^{-120\left((x-\frac12-\frac{1}{5}\cos\frac{m\pi}{3})^2
+(y-\frac12-\frac{1}{5}\sin\frac{m\pi}{3})^2\right)}
\Big\},
\end{split}
\end{equation*}
\begin{equation*}
\begin{split}
f(x,y,t)
=
e^{-\lambda t}
\bigl\{
&
\left(
-t^{\alpha-1}E_{\alpha,\alpha}(-t^\alpha)
-(\lambda+1)E_{\alpha,1}(-t^\alpha)
\right)\phi(x,y)
\\[3pt]
&-
E_{\alpha,1}(-t^\alpha)
\left(
\phi_{xx}(x,y)+\phi_{yy}(x,y)-\phi_x(x,y)-\phi_y(x,y)
\right)
\bigr\}.\\[3pt]
\end{split}
\end{equation*}
The corresponding exact solution is $u(x,y,t)=e^{-\lambda t}E_{\alpha,1}(-t^\alpha)\phi(x,y)$. In Fig.~\ref{fig:time_plot3}, convergence results are presented, where our scheme achieves second-order convergence in time, consistent with Theorems~\ref{thm:semi} and~\ref{fully_converge}. In Fig.~\ref{fig:2d_space2}, the errors are shown for varying the polynomial degree \(N\), where the spectral accuracy is again achieved. This validates the convergence result shown in Theorem~\ref{fully_converge}. Moreover, similar to {\bf Example 2}, the effects of the fractional parameters $\alpha$ and $\lambda$ on the solution of \eqref{Equation} are presented in Figs.~\ref{fig:2d_lambda2}-\ref{fig:2d_vary_alpha2}, where similar results have also been observed and obtained. For simplicity, we do not repeat them here.

\begin{figure}[htbp]
	\begin{minipage}[t]{0.45\linewidth}
  \centering
  \includegraphics[width=\textwidth]{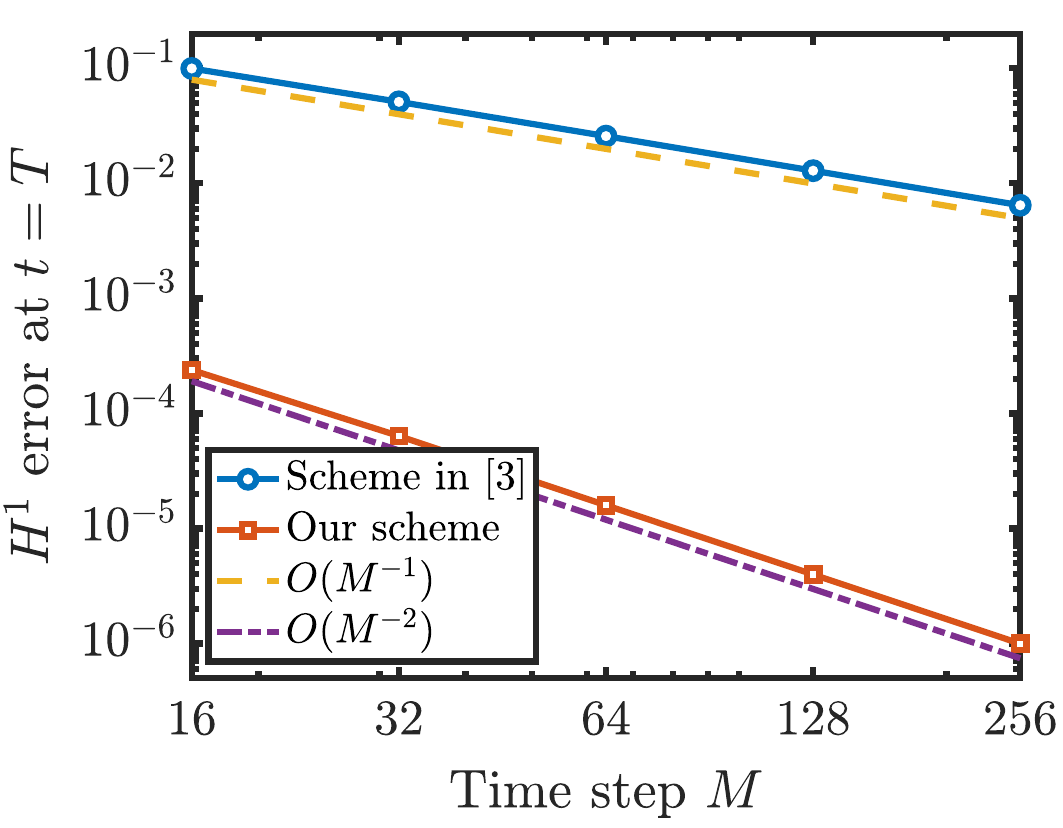}
  \centerline{\qquad\;(a) $\alpha=0.25$}
	\end{minipage}%
    \hspace{0.32cm} 
	\begin{minipage}[t]{0.45\linewidth}
   \centering
   \includegraphics[width=\textwidth]{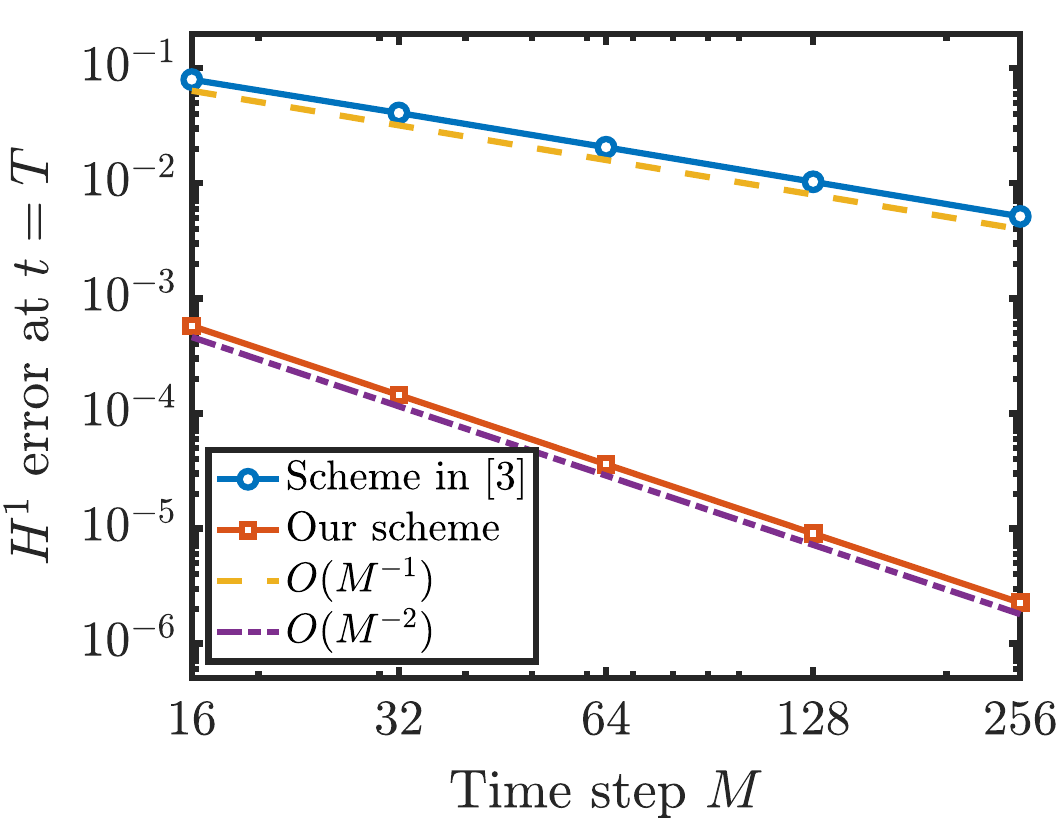}
   \centerline{\qquad\;(b) $\alpha=0.5$}
	\end{minipage}
    \begin{minipage}[t]{0.45\linewidth}
   \centering
   \includegraphics[width=\textwidth]{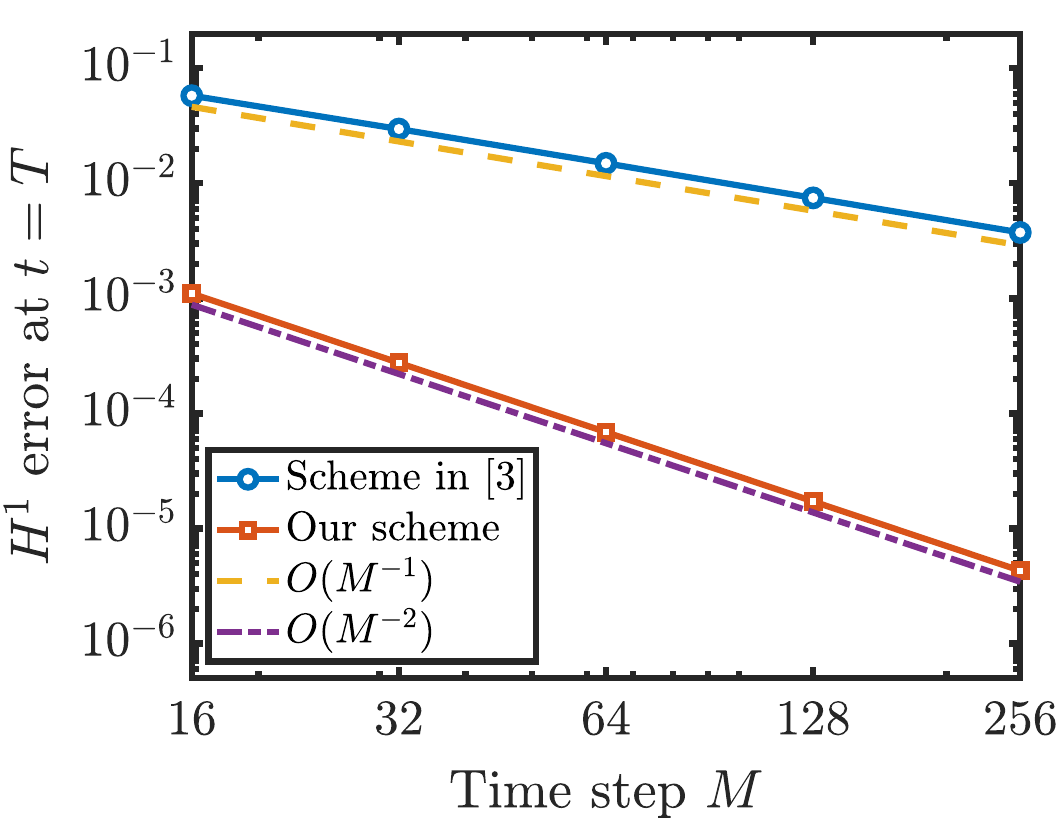}
   \centerline{\qquad\;(c) $\alpha=0.75$}
	\end{minipage}
    \hspace{0.2cm} 
    \begin{minipage}[t]{0.45\linewidth}
   \centering
   \includegraphics[width=\textwidth]{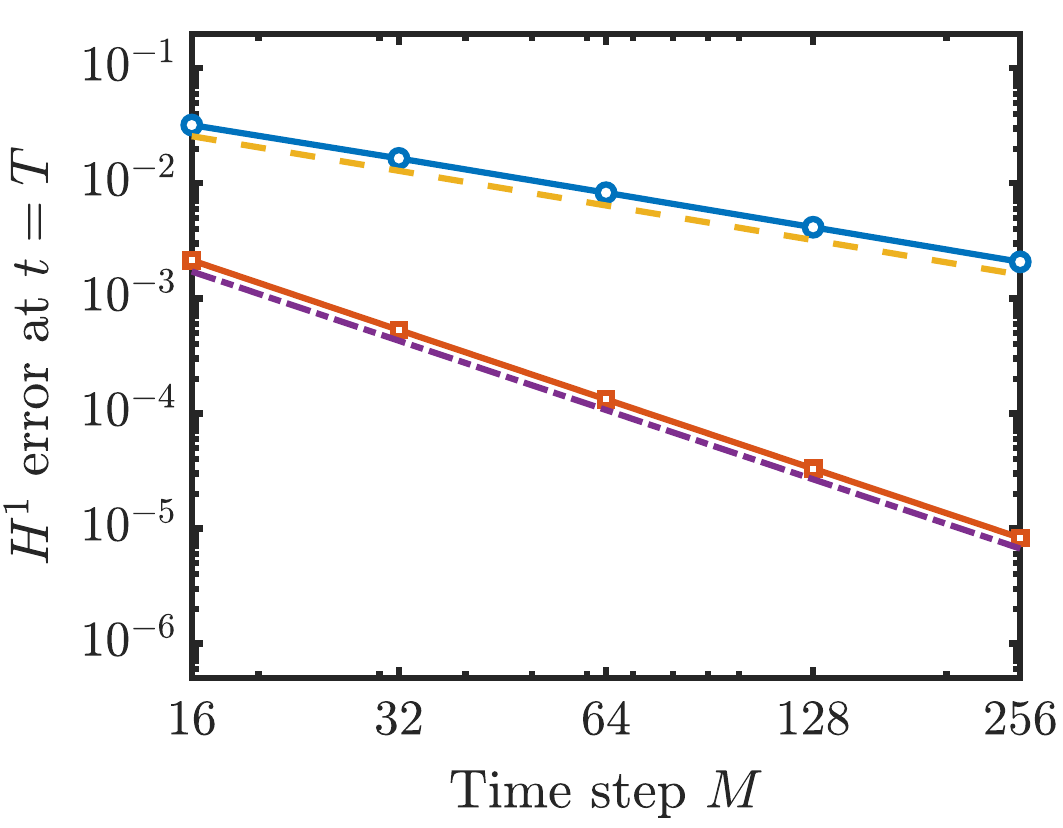}
   \centerline{\qquad\;(d) $\alpha=1.0$}
	\end{minipage}
    \vspace{-0.1cm}
\caption{A comparison between our scheme \eqref{eq:fast_scheme_weak} and the scheme in \cite{cao2020} for {\bf Example 3}.}\label{fig:time_plot3}\vspace{-0.7cm}
\end{figure}  
\begin{figure}[htbp]
    \centering
\includegraphics[width=0.5\textwidth]{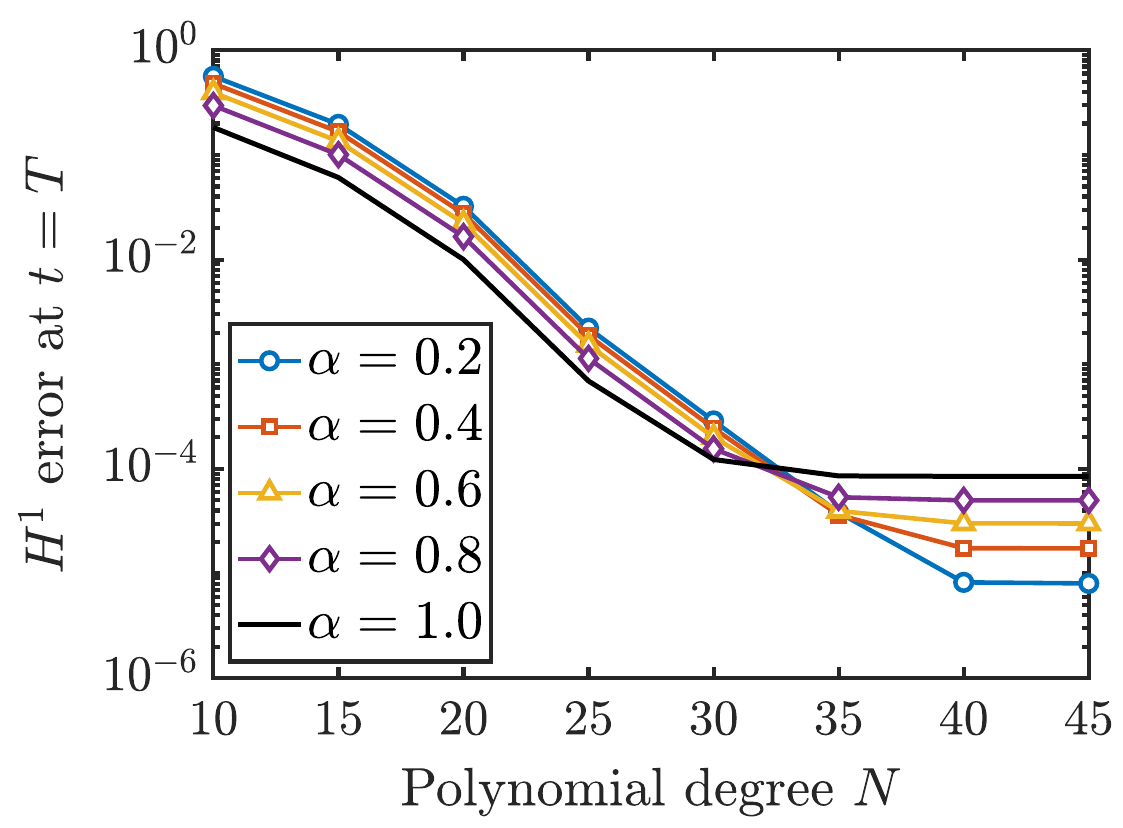}
\vspace{-0.2cm}
    \caption{Errors vs. polynomial degree $N$ for our scheme \eqref{eq:fast_scheme_weak} for {\bf Example 3}.}
    \label{fig:2d_space2}
\end{figure}

\begin{figure}[htbp]
    \centering
\includegraphics[width=0.92\textwidth]{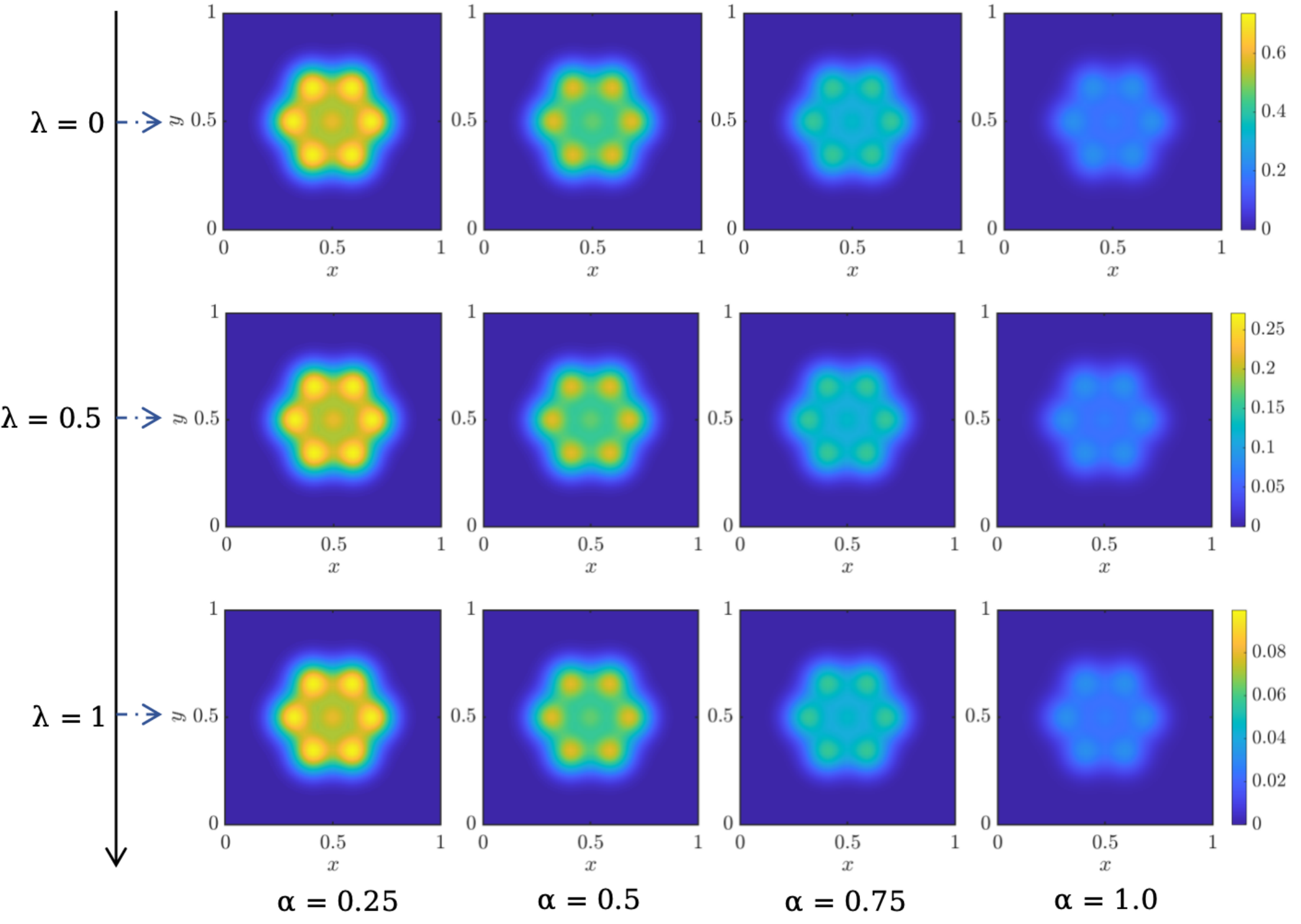}
\vspace{-0.2cm}
    \caption{Numerical solutions of \eqref{Equation} with varying $\lambda$ and $\alpha$ for {\bf Example 3}.} \vspace{-0.1cm}
    \label{fig:2d_lambda2}
\end{figure}

\begin{figure}[htbp]
    \centering
\includegraphics[width=0.92\textwidth]{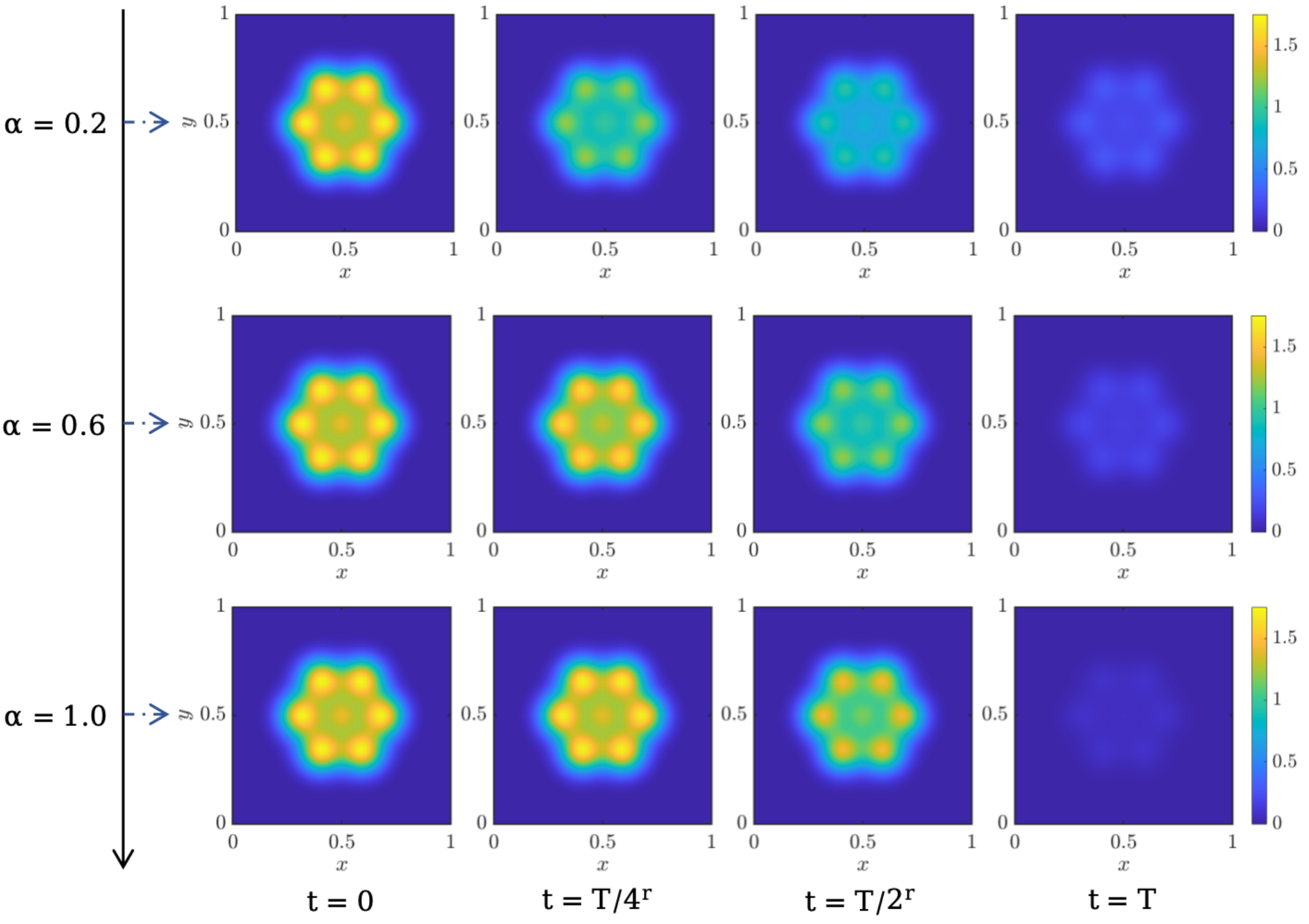}
\vspace{-0.2cm}
    \caption{Numerical solutions of \eqref{Equation} with varying $\alpha$ for {\bf Example 3}.}\vspace{-0.7cm}
    \label{fig:2d_vary_alpha2}
\end{figure}

\section{Concluding remarks}\label{sec:conclude}
In this paper, we have presented a second-order $H^1$-norm error analysis for the time-fractional advection-dispersion equations (TFADE). An efficient integrating-factor transformation is introduced, and error estimates for the discrete coefficients and truncation errors of the fast averaged L1 method on a graded temporal mesh are established. Numerical examples confirm the theoretical results in Theorems~\ref{thm:semi} and~\ref{fully_converge}. Future work will extend this framework to more complex problems, including variable-order and nonlinear fractional equations.

\bibliographystyle{h1/siamplain.bst}
\bibliography{h1/reference.bib}
\end{document}